\documentclass [12pt]{amsart}
\usepackage[utf8]{inputenc}
\pdfoutput=1
\usepackage{amsmath,amssymb,amsthm,amsfonts}
\usepackage{mathtools}
\DeclareFontFamily{U}{mathx}{\hyphenchar\font45}
\DeclareFontShape{U}{mathx}{m}{n}{
      <5> <6> <7> <8> <9> <10>
      <10.95> <12> <14.4> <17.28> <20.74> <24.88>
      mathx10
}{}
\DeclareSymbolFont{mathx}{U}{mathx}{m}{n}
\DeclareMathAccent{\widecheck}{0}{mathx}{"71}

\usepackage[english]{babel}
\usepackage{comment}
\usepackage{hyperref}
\usepackage{bbm}
\usepackage{mathrsfs}

\usepackage{tikz,graphicx,color}
\usepackage{tikz-cd}
\usepackage{tikz-3dplot}
\usetikzlibrary{calc}
\usetikzlibrary{arrows}
\usetikzlibrary{shapes}
\usetikzlibrary{patterns}
\usetikzlibrary{positioning}
\usetikzlibrary{arrows.meta}
\usetikzlibrary{decorations.markings}
\usetikzlibrary{knots}

\usepackage{epstopdf}

\usepackage[arrow]{xy}
\usepackage{diagbox}
\usepackage{subfig}
\usepackage{arcs}
\usepackage{xcolor}
\usepackage{xspace}

\usepackage[margin=1.0in]{geometry}

\usepackage{enumitem}
\usepackage[b]{esvect}
\usepackage{changepage}
\usepackage{letltxmacro}
\usepackage{thmtools,etoolbox}

\def\myarabic#1{\normalfont(\roman{#1})}
\newlist{theoremlist}{enumerate}{1}
\setlist[theoremlist]{label=\myarabic{theoremlisti},ref={\myarabic{theoremlisti}},itemindent=0pt,labelindent=0pt,
  leftmargin=*,noitemsep}

\makeatletter
\renewcommand{\p@theoremlisti}{\perh@ps{\thetheorem}}
\protected\def\perh@ps#1#2{\textup{#1#2}}
\newcommand{\itemrefperh@ps}[2]{\textup{#2}}
\newcommand{\itemref}[1]{\begingroup\let\perh@ps\itemrefperh@ps\ref{#1}\endgroup}
\makeatother

\usepackage{nameref,hyperref}
\usepackage[capitalize]{cleveref}

\newtheorem{theorem}{Theorem}[section]

\newtheorem{lemma}[theorem]{Lemma}
\newtheorem{proposition}[theorem]{Proposition}
\newtheorem{corollary}[theorem]{Corollary}
\theoremstyle{definition}
\newtheorem{remark}[theorem]{Remark}
\theoremstyle{definition}
\newtheorem{definition}[theorem]{Definition}
\newtheorem{conjecture}[theorem]{Conjecture}

\theoremstyle{definition}

\theoremstyle{definition}
\newtheorem{example}[theorem]{Example}

\crefname{figure}{Figure}{Figures}

\def\figref#1(#2){Figure~\hyperref[#1]{\ref*{#1}(#2)}}

\addtotheorempostheadhook[theorem]{\crefalias{theoremlisti}{theorem}}
\addtotheorempostheadhook[lemma]{\crefalias{theoremlisti}{lemma}}
\addtotheorempostheadhook[proposition]{\crefalias{theoremlisti}{proposition}}
\addtotheorempostheadhook[corollary]{\crefalias{theoremlisti}{corollary}}

% ...

\def\R{{\mathbb{R}}}

\def\P{{\mathbb{P}}}

% ...

% ...

\def\<{{\langle}}
\def\>{{\rangle}}

\def\Im{{ \operatorname{Im}}}

\def\Povtp_#1{\Pi_{#1}^{>0}}
\def\Povtnn_#1{\Pi_{#1}^{\geq0}}

\newcounter{todofigure}

\numberwithin{equation}{section}

\makeatletter
\@namedef{subjclassname@2020}{\textup{2020} Mathematics Subject Classification}
\makeatother

\begin{document}
\numberwithin{equation}{section}

\title{Homomesy of Tropical $T$-systems: Finite Type}
\author{Ariana Chin and Pavlo Pylyavskyy}
\address{Department of Mathematics, University of California, Los Angeles, CA 90095, USA}
\email{{\href{mailto:arianagchin@math.ucla.edu}{arianagchin@math.ucla.edu}}}
\thanks{P.P. was supported by a grant from Simons Foundation International SFI-MPS-TSM-00013120}
\date{\today}

\subjclass[2020]{
  Primary:
  13F60, 
  Secondary:
  05E99, 
  22E46
}

\keywords{Cluster algebras, Zamolodchikov periodicity, $T$-systems, homomesy, tropical totally nonnegative Grassmannian}

\begin{abstract}
In 2007, Fomin and Zelevinsky introduced the \emph{bipartite belt}, a sequence of bipartite cluster mutations whose exchange relations form a discrete dynamical system.
For each Dynkin diagram, this system is periodic, as a special case of \textit{Zamolodchikov periodicity}. 
In the associated tropical dynamics, every mutation along the belt can be naturally colored red or blue, according to which term attains the maximum in the tropical exchange relation. This gives a red-blue statistic on tropical orbits. We prove that this statistic is homomesic, i.e. that the average numbers of red and blue mutations are independent of the orbit. 
Moreover, the number of red mutations is given by the number of roots in the associated root system.
This resolves a conjecture of the first author. In the process we introduce the \emph{Zamolodchikov fan}, which is closely related to the positive tropical Grassmannian, and may be of independent interest. 

\end{abstract}

\maketitle

\setcounter{tocdepth}{1}
\tableofcontents

\section{Introduction}
\subsection{Bipartite recurrent quivers}

In their influential paper \cite{fomin1}, Fomin and Zelevinsky introduced cluster algebras. Since then, cluster algebras have become ubiquitous, with applications and connections to representation theory, Teichm\"uller and Poisson geometry, integrable systems, mirror symmetry, and mathematical
physics; see, for example, \cite{KellerPeriodicity,FockGoncharov,GSV,sign-coherence,KunibaNakanishiSuzuki,Williams}. 

In the skew-symmetric case, the exchange matrices defining a cluster algebra can be encoded by {\it quivers}.
A quiver $Q=(V,E)$ is a directed graph without loops or directed $2$-cycles. For each $v\in V$ there is a {\it quiver mutation} $\mu_v$, which changes $Q$ locally at $v$. At the same time, one associates a variable $x_v$ to each vertex, and mutation at $v$ replaces $x_v$ according to the corresponding {\it exchange relation}. Iterating these mutations produces the {\it cluster variables} of the cluster algebra.

We will be interested in quivers whose underlying graph is bipartite. Fix a proper coloring
\[
V=V_{\bullet}\sqcup V_{\circ}.
\]
Since no two vertices of the same color are adjacent, the corresponding mutations commute, and we can define
\[
\mu_{\bullet}=\prod_{v\in V_{\bullet}}\mu_v,
\qquad
\mu_{\circ}=\prod_{v\in V_{\circ}}\mu_v.
\]
A bipartite quiver $Q$ is called {\it recurrent} if
\[
\mu_{\bullet}(Q)=\mu_{\circ}(Q)=Q^{\operatorname{op}},
\]
where $Q^{\operatorname{op}}$ is obtained from $Q$ by reversing all arrows.

\subsection{Periodic $T$-systems}
\label{subsec:t-systems quivers}
Given a bipartite recurrent quiver $Q$, one can associate to it a $T$-system as follows. For each vertex $v\in \operatorname{Vert}(Q)$ and each $t\in \mathbb Z$ of the appropriate parity, we introduce a variable $T_v(t)$. The exchange relations obtained by alternating mutations at black and white vertices take the form
\[
T_v(t+1)T_v(t-1)
=
\prod_{u\to v} T_u(t)
+
\prod_{v\to w} T_w(t),
\]
where the first product is over all arrows entering $v$, and the second product is over all arrows leaving $v$. Tropicalizing this recurrence, products become sums and addition becomes taking the maximum. Thus the corresponding {\it tropical $T$-system} is
\[
\boldsymbol{t}_v^{\lambda}(t+1)+\boldsymbol{t}_v^{\lambda}(t-1)
=
\max\left(
\sum_{u\to v}\boldsymbol{t}_u^{\lambda}(t),
\sum_{v\to w}\boldsymbol{t}_w^{\lambda}(t)
\right).
\]

The sequence of mutations
\[
\mu_{\circ}\mu_{\bullet}\mu_{\circ}\mu_{\bullet}\cdots
\]
is called the {\it bipartite belt}. Fomin and Zelevinsky studied this dynamics for cluster algebras of finite type in \cite{fomin4}. For certain classes of cluster algebras the bipartite belt is periodic: after finitely many steps, both the quiver and the corresponding cluster variables return to their initial values. This phenomenon is called \textit{Zamolodchikov periodicity}, first observed by Zamolodchikov in his study \cite{zamolodchikov} of thermodynamic Bethe ansatz which initially focused on the simply-laced Dynkin diagrams $A_n, D_n$, and $E_n$.
His conjecture was generalized by Ravanini--Valeriani--Tateo \cite{ravanini}, Kuniba--Nakanishi \cite{kuniba}, Kuniba--Nakanishi--Suzuki \cite{kuniba2}, and Fomin--Zelevinsky \cite{fomin2}.

Of special importance to us are recurrent quivers obtained from pairs of
bipartite Dynkin diagrams. Let $\Lambda$ and $\Lambda'$ be bipartite graphs,
with vertex decompositions
\[
\Lambda=\Lambda_\bullet\sqcup \Lambda_\circ,
\qquad
\Lambda'=\Lambda'_\bullet\sqcup \Lambda'_\circ.
\]
Their {\it tensor product} $\Lambda\otimes\Lambda'$ has vertex set
\[
\operatorname{Vert}(\Lambda\otimes\Lambda')
=
\operatorname{Vert}(\Lambda)\times\operatorname{Vert}(\Lambda').
\]
There is an edge between $(u,w)$ and $(v,w)$ whenever $u$ and $v$ are
adjacent in $\Lambda$, and between $(u,w)$ and $(u,w')$ whenever $w$ and
$w'$ are adjacent in $\Lambda'$. These edges are oriented so that
\[
(\Lambda_\bullet,\Lambda'_\bullet)\to
(\Lambda_\circ,\Lambda'_\bullet)\to
(\Lambda_\circ,\Lambda'_\circ)\to
(\Lambda_\bullet,\Lambda'_\circ)\to
(\Lambda_\bullet,\Lambda'_\bullet).
\]
When $\Lambda$ and $\Lambda'$ are simply-laced Dynkin diagrams, this gives a
bipartite recurrent quiver.

The following famous periodicity result resolves Zamolodchikov's conjecture for tensor products, and was proved in full generality by Keller \cite{KellerPeriodicity}.

\begin{theorem}[Keller]
Let $\Lambda$ and $\Lambda'$ be simply-laced Dynkin diagrams with Coxeter
numbers $h$ and $h'$. Then the $T$-system associated with
$\Lambda\otimes\Lambda'$ is periodic. More precisely,
\[
T_v(t+2(h+h'))=T_v(t)
\]
for every vertex $v$ and every admissible $t$.
\end{theorem}

The result was also independently proven by Inoue--Iyama--Keller--Kuniba--Nakanishi \cite{inoue, inoue2}.
The full classification of Zamolodchikov periodic quivers was completed by Galashin--Pylyavskyy \cite{pasha} and in full generality to skew-symmetrizable families by Chin \cite{chin}.

\subsection{Homomesy and the main theorem}

Homomesy is a phenomenon introduced by Propp and Roby in \cite{ProppRoby}. Let $S$ be a set, and consider a bijective map $\tau:S\rightarrow S$ such that every $\tau$-orbit is finite. Let $f:S\rightarrow \mathbb{R}$ be a statistic. We say that $f$ is {\it homomesic} under $\tau$ if there exists a constant $c\in\mathbb{R}$ such that for every $\tau$-orbit $\mathcal O\subset S$,
\[
\frac{1}{|\mathcal O|}\sum_{x\in\mathcal O} f(x)=c.
\]
In other words, although the values of $f$ may vary along an orbit, their average is the same for every orbit. Homomesy has been studied in promotion and rowmotion, see \cite{ProppRoby,EinsteinPropp}. 

For a quiver $Q$, consider the set
\[
S=\mathbb R^{|V(Q)|}
\]
of possible states of the tropical $T$-system, and define
\[
\tau=\mu_{\circ}\mu_{\bullet}:S\to S.
\]
Thus one application of $\tau$ consists of mutating first at all black vertices and then at all white vertices. Although $S$ is infinite, periodicity of the tropical $T$-system implies that $\tau$ has finite order, and hence every $\tau$-orbit is finite.

Now suppose that $Q$ is the tensor product of two Dynkin quivers $\Gamma$ and $\Delta$. Each tropical mutation can be assigned one of two colors. We call the mutation {\it red} if the maximum in the tropical exchange relation is attained by the term coming from the $\Gamma$-direction, and {\it blue} if it is attained by the term coming from the $\Delta$-direction. For $s\in S$, let $f(s)$ be the number of red mutations that occur during the next application of $\tau$.

The following conjecture was observed in \cite{chin}, and is equivalent to the statement that $f$ is homomesic under $\tau$. 
\begin{conjecture}[Chin]
\label{conj}
    \textit{Let $(\Gamma, \Delta)$ be a Zamolodchikov periodic cluster algebra of rank $n$, and let $h_{\Gamma}$ be the associated Coxeter number to $\Gamma$.
    In one period of the corresponding tropical $T$-system, the number of red mutations is $h_{\Gamma}\cdot n$.}
\end{conjecture}

In this paper we prove this conjecture when one of the two Dynkin quivers consists of a single vertex. Equivalently, we prove the following.

\begin{theorem}
\label{mainThm}
Let $\Lambda$ be a Dynkin diagram of rank $n$. In one period of the corresponding tropical $T$-system, the number of red mutations is the number of roots in the associated root system of $\Lambda$.
\end{theorem}

\subsection{Zamolodchikov  fan}

In the process of proving Theorem \ref{mainThm}, we introduce the notion of a
{\it Zamolodchikov  fan}, which we believe to be of independent
interest. Recall that tropical mutation gives piecewise-linear maps on
$\mathbb R^{|V(Q)|}$. Restricting to the mutation sequences appearing in the
bipartite belt, we consider the common refinement of their domains of
linearity. We call the resulting fan $\mathcal{Z}$ the {\it Zamolodchikov  fan}.

Equivalently, its cones are the maximal regions of initial conditions for
which every tropical mutation along the bipartite belt makes the same choice
in each $\max$. Thus, when $Q=\Gamma\otimes\Delta$, the red/blue mutation
pattern is constant on each cone of the Zamolodchikov  fan.

Tensor products of type $A$ quivers appear naturally in the cluster algebra
structure on Grassmannians. In particular, the quiver
\[
A_{k-1}\otimes A_{n-k-1}
\]
appears as a rectangular seed for $\operatorname{Gr}(k,n)$. 
This connection goes back, in a different language, to Volkov's proof of Zamolodchikov
periodicity for pairs of type $A$ Dynkin diagrams \cite{Volkov}. His argument
uses projective geometry rather than cluster algebras, but the underlying
Grassmannian picture is essentially the same. The cluster algebra structure on Grassmannians was developed systematically by Scott \cite{Scott}, relying on Postnikov's work on the totally nonnegative Grassmannian for its combinatorial framework \cite{Postnikov}.

The \textit{positive tropical Grassmannian} $\operatorname{Trop}^+ G(k,n)$ was
introduced by Speyer and Williams \cite{SpeyerWilliams}. For $k=2$ it is
parametrized by planar trees with $n$ labeled leaves, whose edges
have nonnegative lengths and whose planar embedding is compatible with the
cyclic order $1,2,\ldots,n$. Dually, its cones are indexed by
collections of pairwise noncrossing diagonals of an $n$-gon. In particular,
$\operatorname{Trop}^+G(2,n)$ is combinatorially the normal fan of the
type $A_{n-3}$ associahedron.

More generally, one can associate a mutation fan to any collection of
cluster mutations by taking the common refinement of their domains of
linearity. For $G(k,n)$, taking all mutations between Pl\"ucker clusters gives a
natural {\it Pl\"ucker mutation fan}. 
By the weak separation description
of Pl\"ucker clusters due to Oh, Postnikov, and Speyer \cite{OPS}, every
short Pl\"ucker relation occurs as such a mutation. Hence the Pl\"ucker mutation fan is precisely the positive Dressian, which, by a theorem of
Speyer and Williams \cite{dressian}, coincides with
$\operatorname{Trop}^+G(k,n)$.

For $G(2,n)$, the smaller collection of mutations appearing on the bipartite belt already
recovers the whole fan.

\begin{theorem}
\label{thm:ZamTropGr}
The Zamolodchikov  fan associated with $G(2,n)$ coincides with the
positive tropical Grassmannian $\operatorname{Trop}^+G(2,n)$.
\end{theorem}

Already for $G(3,6)$ the two fans are different however, see Remark \ref{rem:diff}.

\;

The paper proceeds as follows. In Section~\ref{prelims} we introduce cluster algberas, laminations, and (tagged) triangulations as a cluster algebra. In Section~\ref{sec:fans} we introduce the Zamolodchikov fan $\mathcal{Z}$ and related fans, and enumerate the number of cones in $\mathcal{Z}$. In Section~\ref{sec:A} we prove Theorem $\ref{mainThm}$
in types $A$ and $C$, and in the process prove Theorem \ref{thm:ZamTropGr} as well.
In Section~\ref{sec:D}, we prove Theorem $\ref{mainThm}$ in types $B$ and $D$. Finally, in Section~\ref{sec:E} we describe the computer verification in the
exceptional types $E_6, E_7, E_8, F_4,$ and $G_2$.

\subsection{Acknowledgments and AI statement}

We are grateful to Pavel Galashin for his feedback and support.
AI was used for mild editing; all mathematical results belong to the authors.

\section{Preliminaries}
\label{prelims}

\subsection{Introduction to cluster algebras beyond quivers}
Cluster algebras were first introduced in \cite{fomin1} by Fomin and Zelevinsky with applications to dual canonical bases, total positivity in semisimple Lie groups, and Zamolodchikov periodicity for $Y$-systems of finite root systems.

An $n\times n$ matrix $B$ is \textit{skew-symmetrizable} if there exists some $c\in \mathbb{R}_{>0}^n$ such that $\forall i, j\in [n]$, $c_ib_{ij} = -c_jb_{ji}$.
In other words, a matrix $B$ is skew-symmetrizable if there exists some nonzero scaling of the rows $c\in \mathbb{R}_{>0}^n$ that makes the matrix skew-symmetric.
An $m\times n$ matrix, with $m\geq n$, whose top $n\times n$ submatrix is skew-symmetrizable is called an \textit{extended skew-symmetrizable} matrix.

\begin{definition}
    A \textit{labeled seed} in field $\mathcal{F}$ is a pair $(\tilde{\boldsymbol{x}}, \tilde{B})$ where
    \begin{itemize}
        \item 
            $\tilde{\boldsymbol{x}} = (x_1, \dots, x_m)$ is an $m$-tuple of elements in $\mathcal{F}$ forming a \textit{free generating set}; that is, $x_1, \dots, x_m$ are algebraically independent, and $\mathcal{F} = \mathbb{C}(x_1, \dots, x_m)$;
        \item 
            $\tilde{B} = (b_{ij})$ is an $m\times n$ extended skew-symmetrizable integer matrix.
    \end{itemize}
\end{definition}
We use the following notation:
\begin{itemize}
    \item $\tilde{\boldsymbol{x}}$ is the (labeled) \textit{extended cluster} of the labeled seed $(\tilde{\boldsymbol{x}}, \tilde{B})$;
    \item 
        $\boldsymbol{x}$ is the (labeled) \textit{cluster} of this seed;
    \item 
        the elements $x_1, \dots, x_n$ are its \textit{cluster variables};
    \item 
        the remaining elements $x_{n+1}, \dots, x_m$ are its \textit{frozen variables} (or \textit{coefficient variables});
    \item 
        the matrix $\tilde{B}$ is the \textit{extended exchange matrix}, also known as the $B$-matrix;
    \item 
        its top $n\times n$ submatrix $B$ is the \textit{exchange matrix}.
\end{itemize}
A \textit{quiver} is a directed graph without loops and cycles of length 2.
When an exchange matrix $B$ is skew-symmetric, the seed can be represented as a quiver $Q$, where $B$ is the adjacency matrix of $Q$.

\begin{definition}
    Let $(\tilde{\boldsymbol{x}}, \tilde{B})$ be a labeled seed. 
    Let $k\in [n]$.
    The \textit{seed mutation} $\mu_k$ in direction $k$ transforms $(\tilde{\boldsymbol{x}}, \tilde{B})$ into the new labeled seed $\mu_k(\tilde{\boldsymbol{x}}, \tilde{B}) = (\tilde{\boldsymbol{x}}', \tilde{B'})$ defined as follows:
    \begin{itemize}
        \item 
            The extended cluster $\tilde{\boldsymbol{x}}' = (x_1', \dots, x_m')$ is given by $x_j' = x_j$ for all $j\neq k$, whereas $x_k'$ is determined by the \textit{exchange relation}
            \begin{equation}
            \label{mutation}
                x_kx_k' = \prod_{b_{ik} > 0} x_i^{|b_{ik}|} + \prod_{b_{ik} < 0} x_i^{|b_{ik}|}.
            \end{equation}
        \item  
            The extended exchange matrix $\tilde{B'} = (b_{ij}')$ is given by
            $$b_{ij}' = \begin{cases}
                -b_{ij} & \text{if }i = k \text{ or } j = k,\\
                b_{ij} + b_{ik}b_{kj} & \text{if } b_{ik}, b_{kj} > 0,\\
                b_{ij} - b_{ik}b_{kj} & \text{if } b_{ik}, b_{kj} < 0,\\
                b_{ij} & \text{otherwise.}
            \end{cases}$$
    \end{itemize}
\end{definition}

\begin{definition}
    Let $(\tilde{\boldsymbol{x}}, \tilde{B})$ be a seed, and let $\mathcal{X}$ be the set of all cluster variables obtained from any sequence of mutations of this seed.
    The \textit{cluster algebra} $\mathcal{A}$ of geometric type over $\mathbb{C}$ generated by the seed $(\tilde{\boldsymbol{x}}, \tilde{B})$ is the $\mathbb{C}$-subalgebra of the ambient field $\mathbb{C}(\tilde{\boldsymbol{x}})$ generated by all cluster variables: $\mathcal{A} = \mathbb{C}[\mathcal{X}]$.
\end{definition}

In general, $\mathcal{X}$ is an infinite set. When $\mathcal{X}$ is finite, we say the cluster algebra $\mathcal{A}$ is of \textit{finite type}.
Fomin--Zelevinsky proved in \cite{fomin} that the finite type cluster algebras are given by the Cartan--Killing classification of semisimple Lie algebras and finite root systems.
The addition of arbitrary coefficients does not affect whether or not a cluster algebra is of finite type.

\subsection{Tropical $T$-systems for skew-symmetrizable matrices}
Given a bipartite recurrent $B$-matrix, one can decompose it as a sum of two extended skew-symmetrizable matrices $B = \tilde{\Gamma} + \tilde{\Delta}$, $\tilde{\Gamma} = (\tilde{\Gamma}_{ij})$ and $\tilde{\Delta} = (\tilde{\Delta}_{ij})$, as follows:
\begin{align*}
    \tilde{\Gamma}_{ij} \coloneqq \begin{cases}
        b_{ij} & \text{if } b_{ij} > 0, \epsilon_i = \circ, \epsilon_j = \bullet,\\
        b_{ij} & \text{if } b_{ij} < 0, \epsilon_i = \bullet, \epsilon_j = \circ,\\
        0 & \text{otherwise.}
    \end{cases}\qquad
    \tilde{\Delta}_{ij} \coloneqq \begin{cases}
        b_{ij} & \text{if } b_{ij} > 0, \epsilon_i = \bullet, \epsilon_j = \circ,\\
        b_{ij} & \text{if } b_{ij} < 0, \epsilon_i = \circ, \epsilon_j = \bullet,\\
        0 & \text{otherwise.}
    \end{cases}
\end{align*}
Throughout the paper, we often use these matrices as unsigned matrices, so we define the matrices $\Gamma = (\Gamma_{ij})$ and $\Delta = (\Delta_{ij})$ by
$$\Gamma_{ij} \coloneqq |\tilde{\Gamma}_{ij}|, \hspace{5mm} \Delta_{ij}\coloneqq |\tilde{\Delta}_{ij}|$$
for all $i,j\in[n]$.
One can also associate to $B = \tilde{\Gamma} + \tilde{\Delta}$ a \textit{tropical $T$-system}, an infinite system of algebraic equations, defined as follows.

Let $\lambda\in\mathbb{R}^n$ be a labeling of the $n$ indices.
The \textit{tropical $T$-system} $\boldsymbol{t}^{\lambda}$ associated with $B$ is a family $\boldsymbol{t}_k^{\lambda}(t)\in\mathbb{R}$ of real numbers satisfying the following  relations for all $k\in [n]$ and all $t\in\mathbb{Z}$:
\begin{equation}
\label{eq:tropicalMutation}
    \boldsymbol{t}_k^{\lambda}(t + 1) + \boldsymbol{t}_k^{\lambda}(t - 1) = \max\left(\sum_i\Gamma_{ik}\boldsymbol{t}_i^{\lambda}(t), \sum_j \Delta_{jk}\boldsymbol{t}_j^{\lambda}(t)\right).
\end{equation}
Notice that this is analogous to the quiver definition given in \ref{subsec:t-systems quivers}, where $\Gamma$ consists of arrows oriented $\circ\rightarrow \bullet$ and $\Delta$ consists of arrows oriented $\bullet\rightarrow \circ$.
We only consider the subsystem defined for $k\in [n]$ and $t\in\mathbb{Z}$ satisfying
\begin{equation}
\label{splitTsystem}
    \epsilon_k = \circ\quad\text{and}\quad t\equiv 0\hspace{-2mm}\pmod{2}\qquad \text{or} \qquad \epsilon_k = \bullet\quad \text{and} \quad t\equiv 1\hspace{-2mm}\pmod{2}.
\end{equation}

The tropical $T$-system is set to the following initial conditions:
\begin{align*}
    \boldsymbol{t}_k^{\lambda}(0) = \lambda_k\quad\text{and}\quad \boldsymbol{t}_k^{\lambda}(1) &= \lambda_k, \quad \text{for all } k \text{ satisfying } (\ref{splitTsystem}).
\end{align*}

We say the associated tropical $T$-system $\boldsymbol{t}^{\lambda}$ is \textit{periodic} if there exists a positive integer $N$ such that $\boldsymbol{t}_k^{\lambda}(t + 2N) = \boldsymbol{t}_k^{\lambda}(t)$ for all $k\in[n], t\in\mathbb{Z}$.
Then, the associated $B$-matrix $B = \tilde{\Gamma} + \tilde{\Delta}$ is \textit{Zamolodchikov periodic}.


Let $B$ be a Zamolodchikov periodic $B$-matrix, and let $\boldsymbol{t}_i^{\lambda}$ be its corresponding tropical $T$-system with initial labeling $\lambda\in \mathbb{R}^{m}$.
We can color a tropical mutation (\ref{eq:tropicalMutation}) at index $k\in [n]$ either red or blue, depending on the evaluation of the max function. 
If the maximum is $\sum_i\Gamma_{ik}\boldsymbol{t}_i^{\lambda}(t)$, then this is a \textit{$\Gamma$-mutation}, or a red mutation.
Otherwise if the maximum is $\sum_j\Delta_{jk}\boldsymbol{t}_j^{\lambda}(t)$, then this is a \textit{$\Delta$-mutation}, or a blue mutation.

\subsection{Cluster algebras of type $A_n$: triangulations of an $(n+3)$-gon}

A cluster algebra of type $A_n$ can be geometrically realized by triangulations of an $(n+3)$-gon (see Figure \ref{fig:triangulation ex}).
As presented in \cite{fomin}, the clusters are in bijection with the triangulations of an $(n+3)$-gon by non-crossing diagonals. 
Each frozen variable is represented as one of the $n+3$ sides of the polygon (i.e. $P_{ij}$, where $i,j$ are consecutive vertices).
Each cluster variable is represented as some diagonal $P_{ij}$, where $i,j$ are not consecutive vertices in the polygon.
There are ${n+3\choose 2} - (n+3)$ such nontrivial diagonals, for $\frac{n(n+3)}{2}$ cluster variables.
Given a triangulation, one can uniquely recover the $B$-matrix by constructing the quiver whose arrows form clockwise cycles within each triangle (see Figure \ref{fig:triangulation ex}).
Two triangulations are related via a single mutation if and only if they
are obtained from each other by a ``flip” that replaces a diagonal in a quadrilateral formed by two triangles of the triangulation by another diagonal of the same quadrilateral.
We call this surrounding quadrilateral the \textit{mutation quadrilateral}.
For more details on the geometric realization, we refer the reader to \cite{fomin} and \cite{fomin2}.

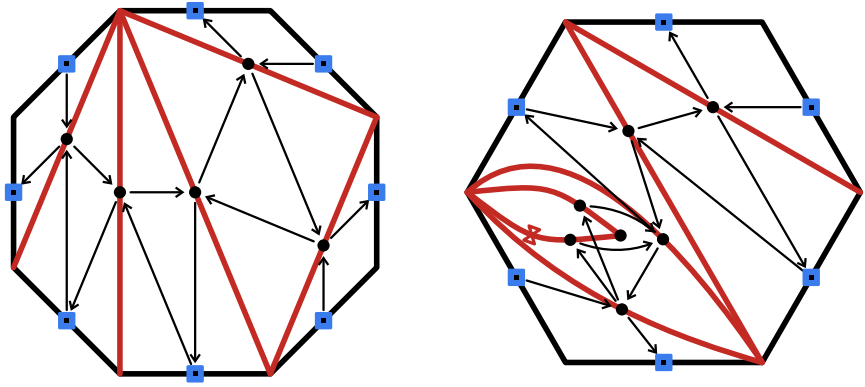
\begin{figure}
\centering
\definecolor{seedred}{RGB}{196,42,36}
\definecolor{seedblue}{RGB}{58,124,235}
\begin{tikzpicture}[
    line cap=round, line join=round,
    bdry/.style={line width=2.1pt, black},
    arc/.style={line width=2.1pt, seedred},
    arr/.style={-{Straight Barb[length=2.6pt,width=3.6pt,line width=0.85pt]},
                line width=0.85pt, shorten >=4pt, shorten <=4pt},
    mut/.style={circle, fill=black, inner sep=0pt, minimum size=4.6pt}
  ]

  \def\R{2.6}     
  \def\Px{6.2}    
  \def\frz#1{\fill[seedblue,rounded corners=0.6pt]
      ($#1+(-0.115,-0.115)$) rectangle ($#1+(0.115,0.115)$);
      \fill[black] ($#1+(-0.032,-0.032)$) rectangle ($#1+(0.032,0.032)$);}
  \def\notchat#1#2{\begin{scope}[shift={(#1)}, rotate=#2]
      \draw[seedred,line width=1.2pt] (0,0) -- (-0.10, 0.075) -- (-0.10,-0.075) -- cycle;
      \draw[seedred,line width=1.2pt] (0,0) -- ( 0.10, 0.075) -- ( 0.10,-0.075) -- cycle;
    \end{scope}}

  \begin{scope}[shift={(0,0)}]
    \foreach \k in {0,...,7}{\coordinate (V\k) at ({22.5+45*\k}:\R);}
    \foreach \k in {0,...,7}{\coordinate (b\k) at ({45+45*\k}:{0.9239*\R});}
    \draw[bdry] (V0) -- (V1) -- (V2) -- (V3) -- (V4) -- (V5) -- (V6) -- (V7) -- cycle;
    \draw[arc] (V2) -- (V4); \coordinate (d1) at ($(V2)!0.5!(V4)$);
    \draw[arc] (V2) -- (V5); \coordinate (d2) at ($(V2)!0.5!(V5)$);
    \draw[arc] (V2) -- (V6); \coordinate (d3) at ($(V2)!0.5!(V6)$);
    \draw[arc] (V2) -- (V0); \coordinate (d4) at ($(V2)!0.5!(V0)$);
    \draw[arc] (V6) -- (V0); \coordinate (d5) at ($(V6)!0.5!(V0)$);
    \foreach \s/\t in {d1/d2, d2/d3, d3/d4, d4/d5, d5/d3,
                       b2/d1, d1/b3, d2/b4, b4/d1, d3/b5, b5/d2,
                       b6/d5, d5/b7, b0/d4, d4/b1}{\draw[arr] (\s) -- (\t);}
    \foreach \k in {0,...,7}{\frz{(b\k)}}
    \foreach \v in {d1,d2,d3,d4,d5}{\node[mut] at (\v) {};}
  \end{scope}

  \begin{scope}[shift={(\Px,0)}]
    \foreach \k in {0,...,5}{\coordinate (W\k) at ({60*\k}:\R);}
    \foreach \k in {0,...,5}{\coordinate (c\k) at ({30+60*\k}:{0.8660*\R});}
    \draw[bdry] (W0) -- (W1) -- (W2) -- (W3) -- (W4) -- (W5) -- cycle;
    \coordinate (P) at (-0.22*\R,-0.22*\R);      
    \draw[arc] (W2) -- (W5);  \coordinate (a1) at ($(W2)!0.32!(W5)$);
    \draw[arc] (W2) -- (W0);  \coordinate (a2) at ($(W2)!0.5!(W0)$);
    \draw[arc] (W3) .. controls (-0.50*\R,0.45*\R) and (0.15*\R,-0.35*\R) .. (W5);
    \coordinate (a3) at (-0.004*\R,-0.239*\R);
    \draw[arc] (W3) .. controls (-0.50*\R,-0.55*\R) and (0.10*\R,-0.75*\R) .. (W5);
    \coordinate (a4) at (-0.213*\R,-0.596*\R);
    \draw[arc] (W3) .. controls (-0.566*\R,0.045*\R) and (-0.566*\R,0.045*\R) .. (P);
    \coordinate (a5) at (-0.427*\R,-0.068*\R);
    \draw[arc] (W3) .. controls (-0.654*\R,-0.265*\R) and (-0.654*\R,-0.265*\R) .. (P);
    \coordinate (a6) at (-0.476*\R,-0.242*\R);
    \notchat{(-0.672*\R,-0.217*\R)}{70}
    \foreach \s/\t in {c2/a1, a1/a3, a3/c2, a1/a2, a2/c5, c5/a1,
                       c0/a2, a2/c1, c3/a4, a4/c4,
                       a3/a4, a4/a5, a4/a6}{\draw[arr] (\s) -- (\t);}
    \draw[arr] (a5) to[bend left=18] (a3);
    \draw[arr] (a6) to[bend right=20] (a3);
    \foreach \k in {0,...,5}{\frz{(c\k)}}
    \foreach \v in {a1,a2,a3,a4,a5,a6}{\node[mut] at (\v) {};}
    \node[mut] at (P) {};
  \end{scope}

\end{tikzpicture}
\caption{\label{fig:triangulation ex} Geometric realizations of seeds of type $A_5$ and $D_6$.  On the left is a triangulation of an octagon and its associated quiver.  On the right is a tagged triangulation of a punctured hexagon, corresponding to $D_6$, along with its associated quiver. Mutable variables are indicated by black vertices, and frozen variables are indicated by blue squares.}
\end{figure}

\subsection{Cluster algebras of type $D_n$: tagged triangulations of a punctured $n$-gon}
\label{subsec:tagged triangulations}

Tagged triangulations were first defined on marked surfaces in \cite{shapiro}.
The following definition uses tagged triangulations only in the context of once-punctured polygons.
For more details, we refer the reader to \cite[Section 5.4]{Williams}.

A once-punctured $n$-gon contains $n+1$ marked points: $n$ vertices on the boundary, and one puncture $p$.
An \textit{arc} is a non-selfintersecting curve $\gamma_{ij}$ whose endpoints $i,j$ are distinct marked points, and which does not cut out an unpunctured bigon.
An arc incident to the puncture is called a radius.
A \textit{tagged arc} is either a regular non-radius arc, or a radius that has been either \textit{tagged} or left \textit{untagged}.

Two arcs $\gamma, \gamma'$ are \textit{compatible} if they do not cross, with the following exception: two radii with distinct endpoints are incompatible if they are tagged differently.
A \textit{tagged triangulation} is a maximal (by inclusion) collection of pairwise compatible tagged arcs.

A cluster algebra of type $D_n$ can be geometrically realized by tagged triangulations of a punctured $n$-gon (see Figure \ref{fig:triangulation ex}).
The clusters are in bijection with tagged triangulations of a punctured $n$-gon by compatible tagged arcs.
Each frozen variable is represented as one of the $n$ sides of the polygon.
Each cluster variable is represented as some tagged arc $\gamma_{ij}$ not isotopic to a side of the $n$-gon.

As in type $A$, given a tagged triangulation, one can uniquely recover the associated $B$-matrix by constructing the quiver whose arrows form clockwise cycles within each triangle.
The constructed quiver for a bigon with one tagged radius and one untagged radius is a special case, and is depicted in Figure \ref{fig:triangulation ex}.
Two tagged triangulations are related via a single mutation if and only if they are obtained from each other by a ``flip" that replaces a tagged arc by another tagged arc compatible with the rest of the tagged triangulation.
This is well-defined; there are only two different tagged arcs that can be interchanged compatibly in any given tagged triangulation.
We also call the bounded region containing the two interchangeable tagged arcs a \textit{mutation quadrilateral}, although it is not always strictly a quadrilateral.

\subsection{Laminations}
\label{sec:laminations}
Lamination spaces were first introduced by William Thurston in his work on the geometry of 3-manifolds \cite{WThurston}, and were later shown to have connections to Teichm\"uller theory by Fock and Goncharov \cite{FockGoncharov2}.
Laminations on marked surfaces in the context of cluster algebras were introduced  by Fomin and Dylan Thurston \cite{FominThurston}. 
A lamination $L$ is a finite collection of pairwise nonintersecting curves on the surface, considered up to isotopy, where two radii that share a puncture are considered to be intersecting. 
Two curves are \textit{compatible} if they can appear in the same lamination $L$ (i.e. are not intersecting).
Curves either connect to unmarked points on the boundary of the surface or spiral into a puncture.
We depict a curve that spirals clockwise into a puncture as an untagged radius, and a curve that spirals counterclockwise as a tagged radius.
Figures \ref{fig:lamination} and \ref{fig:D laminations} include several examples of lamination curves.
In this paper, we allow the following compatibility exception: two lamination radii with the same endpoints (up to isotopy) are compatible if they are tagged differently (i.e. spiral in opposite directions).
For more details of the definition, we refer the reader to \cite{FominThurston}.

Given a lamination $L = (\ell_1, \dots, \ell_{r})$ with weights $(w_1, \dots, w_{r})$ and a seed with triangulation arcs $\gamma_1,\ldots,\gamma_m$, define
\begin{equation}
\label{eq:lamination}
    \lambda=(\lambda_1,\ldots,\lambda_m)\in\mathbb R^m,
    \qquad
    \lambda_i=i(L,\gamma_i),
\end{equation}
where $i(L,\gamma_i)$ is the weighted intersection number of $L$ with
$\gamma_i$.
Precisely, $i(L,\gamma_i)$ is defined as a sum of contributions from all intersections of curves in L with the arc $\gamma_i$.
An intersection with a lamination curve $\ell_j$ contributes $w_j$ to $i(L,\gamma_i)$ if at least one of $\ell_j, \gamma_i$ is not a radius.
If both $\ell_j$ and $\gamma_i$ are radii, such an intersection contributes $w_j$ to $i(L,\gamma_i)$ if and only if the two radii are tagged the same way.

These coordinates satisfy the tropical exchange relations.
Namely, if $\gamma_k$ is flipped to $\gamma_k'$ in an ordinary mutation quadrilateral
whose four sides have coordinates $\lambda_1,\lambda_2,\lambda_3,\lambda_4$,
then
\[
\lambda_k+\lambda_k'
=
\max(\lambda_1+\lambda_2,\lambda_3+\lambda_4).
\]
This fact is known (see \cite[Lemma 9.20]{shapiro} and \cite[Lemma 14.9]{FominThurston}), but is also reproved in Lemma \ref{lem:A laminations}.
The case of self-folded triangles is a bit delicate: while it is treated in
\cite[Lemma 14.15]{FominThurston}, we reprove it in Lemma \ref{lem:trop}
in the form we shall use.

\section{The Zamolodchikov fan}
\label{sec:fans}
Given a $B$-matrix $B$ with an initial labeling $\lambda\in\mathbb{R}^{m}$ of the associated tropical $T$-system, one period of mutations will generically produce some color sequence of red and blue mutations.
We can naturally visualize the space of initial labelings $\lambda\in\mathbb{R}^{m}$ as a complete fan whose cones correspond to the sequence of colored mutations produced by $\lambda$.
This \textit{Zamolodchikov fan} $\mathcal Z_B$ (or just $\mathcal{Z}$) is equivalent to the common refinement of the domains of linearity of the tropical mutation maps appearing along the bipartite belt of $B$.
The cones are common domains of linearity, and are separated by codimension 1 \textit{walls} of non-linearity.
Any point on a wall of $\mathcal{Z}$ represents a nongeneric initial tropical labeling (i.e. contains a mutation that is neither blue nor red), whereas a point inside a chamber of $\mathcal{Z}$ represents a generic tropical labeling.
In the Grassmannian case
$B = A_{k-1}\otimes A_{n-k-1}$, we write this fan as $\mathcal Z_{k,n}$.
An example of $\mathcal{Z}_{2,n}$ is depicted in Figure \ref{fig:A2 fan}.

One may define a slight generalization of the Zamolodchikov fan by allowing all tropical mutation sequences.
The resulting \textit{tropical mutation fan} is the complete fan whose cones are common domains of linearity of $\lambda\in\mathbb{R}^{m}$ over all tropical mutation sequences.

A related fan introduced in \cite{reading} is the \textit{mutation fan}, defined as follows.
Let $B$ be the $(m+1)\times n$ matrix obtained from the $B$-matrix associated to the cluster algebra by appending an additional row $\vec{a}$ to the bottom of it.
Each mutation sequence $\mu_k$ defines a piecewise linear function defined as
\begin{align*}
    f_k: \mathbb{R}^n &\rightarrow \mathbb{R}^n\\
    (a_1, \dots, a_n) &\mapsto \left(\mu_k(B)_{m+1, 1}, \dots, \mu_k(B)_{m+1, n}\right).
\end{align*}
The mutation fan is the complete fan whose cones consist of the common domains of linearity of $f_k$ across all mutation sequences $k$.
This is equivalent to tracking the common domains of linearity in the associated $Y$-system across all mutation sequences $k$.
When restricted to only mutations in the bipartite belt, we call this the \textit{Zamolodchikov mutation fan}.

The following lemma describes when $\mathcal{Z}$ is equivalent to the tropical mutation fan.

\begin{lemma}
\label{finite type fans}
    Let $B$ be the exchange matrix of a finite type cluster algebra. 
    Then $\mathcal{Z}_B$ coincides with the tropical mutation fan associated to $B$.
\end{lemma}
\begin{proof}
    We show the equivalence of these two fans by introducing an intermediary fan.
    Define the operator
    \begin{align*}
        \phi: \{x^d~|~d\in\mathbb{Z}^{m}\} &\rightarrow \mathbb{Z}[x_1, \dots, x_n, \dots, x_{m}]\\
        x^d &\mapsto d_1x_1 + \dots + d_{m}x_{m}
    \end{align*}
    which tropicalizes a monomial.
    Every cluster variable can be written as a Laurent polynomial, or a sum of monomials.
    Given a cluster variable $x^{\alpha_1} + x^{\alpha_2} + \dots + x^{\alpha_{\ell}}$, the piecewise linear function $\max(\phi(x^{\alpha_1}), \dots, \phi(x^{\alpha_{\ell}}))$ defines a complete fan whose cones are the domains of linearity.
    The \textit{cluster variable fan} associated to $B$ is the common refinement over all such fans for each cluster variable in the cluster algebra associated to $B$.

    We first show the cluster variable fan coincides with the tropical mutation fan associated to $B$.
    It is important to notice that the function $\max(\phi(x^{\alpha_1}), \dots, \phi(x^{\alpha_{\ell}}))$ for cluster variable $x_k$ satisfies the same exchange relations as the tropical system (\ref{eq:tropicalMutation}).

    Suppose there is a wall $\gamma$ in the tropical mutation fan that is not in the cluster variable fan.
    Then, $\gamma$ is generated by nonlinearity in some tropical mutation of the form 
    $$x_k + x_k' = \max\left(\sum_i \Gamma_{ik}x_i, \sum_i \Delta_{ik}x_i\right).$$
    Since every cluster variable acts linearly on $\gamma$, locally the left hand side is a sum of two linear functions, which is linear, while the right hand side is not linear by assumption.
    This is a contradiction.
    So, some cluster variable must act nonlinearly on $\gamma$, and $\gamma$ appears in the cluster variable fan.
    Thus, the cluster variable fan is a refinement of the tropical mutation fan.

    Now suppose there is a wall $\gamma$ in the cluster variable fan that is not in the tropical mutation fan.
    Then, $\gamma$ is generated by nonlinearity in some cluster variable, which is obtainable from the initial seed by some sequence of mutations.
    Let the cluster variable $x_k'$ be the first such cluster variable to act nonlinearly at $\gamma$ in this mutation sequence.
    Then, $x_k'$ is of the form
    $$x_k' = \max\left(\sum_i \Gamma_{ik}x_i, \sum_i \Delta_{ik}x_i\right) - x_k$$
    which is not linear, while $x_k$ is linear.
    Thus the max function must not be linear on $\gamma$, which generates a wall in the tropical mutation fan.
    So, the tropical mutation fan is a refinement of the cluster variable fan.

    The same argument can be used to show that $\mathcal{Z_B}$ coincides with the restricted cluster variable fan generated by only cluster variables appearing in the bipartite belt.
    In finite type, \cite{fomin2} proved that the bipartite belt achieves every cluster variable in the cluster algebra.
    It follows that in finite type, $\mathcal{Z}_B$ is equivalent to the full cluster variable fan, and therefore also the tropical mutation fan.
\end{proof}

\begin{figure}
\centering
\definecolor{zamred}{RGB}{200,45,40}
\begin{tikzpicture}[
    x=0.78cm, y=0.78cm,
    line cap=round, line join=round,
    ray/.style={line width=2.4pt, black},
    poly/.style={line width=2.8pt, zamred},
    lbl/.style={font=\bfseries, inner sep=1pt}
  ]

  \draw[ray] (-4,0) -- (4,0);      
  \draw[ray] (0,-4) -- (0,4);      
  \draw[ray] (0,0) -- (4,4);       

  \draw[poly] (-1.7,2.7) -- (0.6,2.7) -- (2.5,0.8) -- (2.5,-1.7)
              -- (-1.7,-1.7) -- cycle;

  \node[lbl] at (-2.50, 3.40) {RRRBB};
  \node[lbl] at ( 1.40, 3.40) {RRBBR};
  \node[lbl] at ( 3.65, 1.15) {RBBRR};
  \node[lbl] at ( 2.65,-2.45) {BBRRR};
  \node[lbl] at (-2.50,-2.35) {BRRRB};

\end{tikzpicture}
\caption{\label{fig:A2 fan} The Zamolodchikov fan $\mathcal{Z}_{2, 5}$ associated to type $A_2$.  Each chamber is defined by a unique sequence of red and blue mutations.  This fan can be represented geometrically as a polytope (in red) whose vertices are the cones, and whose facets are lamination curves contributing to each incident vertex.}
\end{figure}
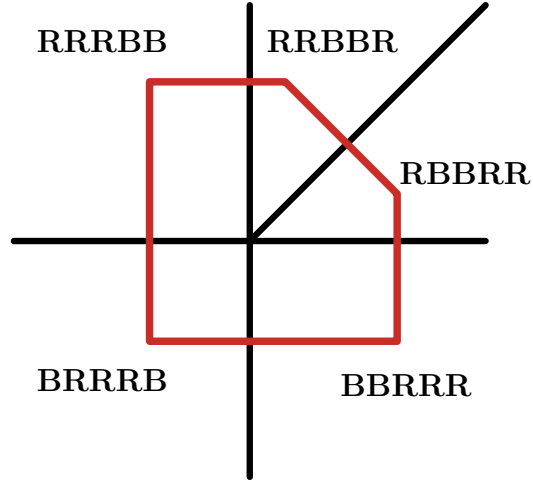

\begin{example}
    Let us examine $\mathcal{Z}_{2, 5}$, the Zamolodchikov fan for $A_2$, which has five seeds and five cluster variables.
    The standard $B$-matrix with trivial coefficients is 
    $B = \begin{bmatrix}
        0 & 1\\
        -1 & 0
    \end{bmatrix},$ which is full rank.
    The Coxeter number for $A_2$ is $h = 3$, so the length of one period is 5.
    If we use the bipartite belt $\mu_1\mu_2\mu_1\mu_2\mu_1$, we obtain all five cluster variables.
    Using the cluster variable fan described above, the cones of $\mathcal{Z}_{2, 5}$ are the common domains of linearity of the following five functions.
    \begin{align*}
        x_1 &\mapsto \max(x_1)\\
        x_2 &\mapsto \max(x_2)\\
        \frac{x_2 + 1}{x_1} &\mapsto \max(x_2 - x_1, -x_1)\\
        \frac{x_1 + x_2 + 1}{x_1x_2} &\mapsto \max(-x_2, -x_1, -x_1 - x_2)\\
        \frac{x_1 + 1}{x_2} &\mapsto \max(x_1 - x_2, -x_2)
    \end{align*}
    The resulting fan is depicted in Figure \ref{fig:A2 fan}.
    Notice that there are five cones, which correspond to the five seeds in the cluster algebra of type $A_2$.
\end{example}

The following proposition describes the relationship between the Zamolodchikov fan and the mutation fan in finite type cluster algebras.

\begin{figure}
\centering
\definecolor{fanblue}{RGB}{38,92,204}
\definecolor{fanred}{RGB}{198,32,32}
\begin{tikzpicture}[
   scale=0.85, every node/.append style={transform shape},
    >={Stealth[round,length=2.6mm,width=2.2mm]},
    line cap=round,
    rule/.style={line width=0.9pt, black!80},
    cell/.style={align=center, font=\normalsize, inner sep=2pt},
    head/.style={font=\bfseries},
    arr/.style={->, line width=1pt, fanblue},
    arrlbl/.style={font=\small, fanblue, inner sep=0pt},
    eqlbl/.style={font=\small, fanred, inner sep=1.5pt, fill=white}
  ]

  \def\xleft{0}        
  \def\xmid{6.4}       
  \def\xright{12.8}    
  \def\ytop{0}         
  \def\ymid{-3.0}      
  \def\ybot{-5.9}      
  \def\yhead{0.50}     
  \def\ycellA{-1.50}   
  \def\ycellB{-4.5}    
  \def\yB{-2.50}       

  \pgfmathsetmacro{\xcolA}{(\xleft+\xmid)/2}
  \pgfmathsetmacro{\xcolB}{(\xmid+\xright)/2}

  \draw[rule] (-1.1,\ytop) -- (\xright,\ytop);
  \draw[rule] (-1.1,\ymid) -- (\xright,\ymid);
  \draw[rule] (\xleft,1.1)  -- (\xleft,\ybot);
  \draw[rule] (\xmid ,1.1)  -- (\xmid ,\ybot);

  \node[head] at (\xcolA,\yhead) {\scalebox{1.3}{Bipartite Belt}};
  \node[head] at (\xcolB,\yhead) {\scalebox{1.3}{All Mutations}};
  \node[head,rotate=90] at (-0.5,{(\ytop+\ymid)/2}) {\scalebox{1.3}{$T$-system}};
  \node[head,rotate=90] at (-0.5,{(\ymid+\ybot)/2}) {\scalebox{1.3}{$Y$-system}};

  \node[cell] (ZF)  at (\xcolA,\ycellA) {\scalebox{1.3}{Zamolodchikov fan}};
  \node[cell] (TMF) at (\xcolB,\ycellA) {\scalebox{1.3}{Tropical Mutation fan}};
  \node[cell] (ZMF) at (\xcolA,\ycellB) {\scalebox{1.3}{Zamolodchikov}\\\scalebox{1.3}{Mutation fan}};
  \node[cell] (MF)  at (\xcolB,\ycellB) {\scalebox{1.3}{Mutation fan}};

  \draw[fanred,line width=0.9pt] (\xmid-0.65,\ycellA+.055) -- (\xmid+0.35,\ycellA+0.055);
  \draw[fanred,line width=0.9pt] (\xmid-0.65,\ycellA-0.055) -- (\xmid+0.35,\ycellA-0.055);
  \node[eqlbl,above] at (\xmid,\ycellA - 0.75) {in finite type};

  \draw[arr] (\xcolA, -2)  -- (ZMF.north);
  \draw[arr] (\xcolB, -2) -- (MF.north);
  \node[arrlbl,anchor=west] at (\xcolA+0.18,\yB) {$B$};
  \node[arrlbl,anchor=west] at (\xcolB+0.18,\yB) {$B$};

\end{tikzpicture}
\caption{\label{fig:fan table} The relationship between all fans defined in Section \ref{sec:fans}.}
\end{figure}
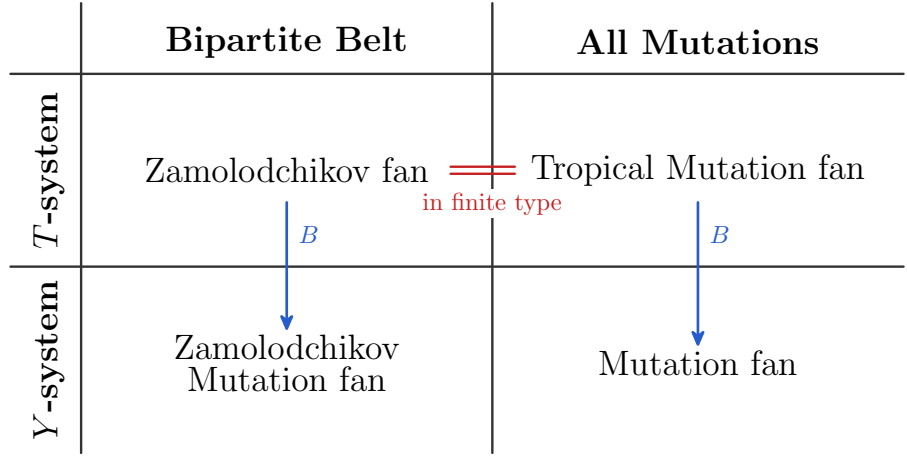

\begin{proposition}
\label{mutation fan}
    Let $\Lambda$ be a finite type cluster algebra of rank $n$, and let $B$ be an (extended) exchange matrix associated to $\Lambda$.
    Then $\mathcal{Z}_B$ is combinatorially equivalent to a subspace $\Im(B)$ of the mutation fan.
\end{proposition}
\begin{proof}
    By Lemma \ref{finite type fans}, the Zamolodchikov fan is equivalent to the tropical mutation fan, so in this proof we examine only the latter.
    Given a labeling $\lambda\in\mathbb{R}^{m}$ in the tropical $T$-system, multiplying by $B^{\top}$ on the left results in $B^{\top}\lambda\in\mathbb{R}^n$, a labeling in the $Y$-system.
    It is straightforward to check that applying $B^{\top}$ commutes with mutation, so $B^{\top}$ takes in the $T$-system $\boldsymbol{t}^{\lambda}$, and outputs the $Y$-system $\boldsymbol{y}^{B^{\top}(\lambda)}$.

    Let $\{\boldsymbol{y}_i~|~i\in[n]\}$ be values in the $Y$-system.
    A nonlinearity is formed by mutation at $k$ exactly when $\boldsymbol{y}_k = 0$.
    Let $\lambda = (\lambda_1, \dots, \lambda_m)$ be values in the $T$-system.
    In the $T$-system, a nonlinearity is formed by mutation at $k$ exactly when
    \begin{align*}
        0 &= \sum_i \Gamma_{ik}\lambda_i - \sum_i \Delta_{ik}\lambda_i = \sum_i (\Gamma_{ik} - \Delta_{ik})\lambda_i,
    \end{align*}
    or equivalently when $(B^{\top}\lambda)_k = 0$.
    Thus when $\boldsymbol{y}_i = (B^{\top}\lambda)_i$ for all $i\in [n]$, the nonlinearities formed in the $T$-system exactly correspond to those formed in the $Y$-system.
    
    The mutation fan is formed by the domains of linearity of the $Y$-system.
    So, the walls of the tropical mutation fan correspond to the walls of the $Y$-system when restricted to initial labelings in $\Im(B)$.
    Thus the tropical mutation fan and $\mathcal{Z}_B$ are combinatorially equivalent to the subspace $\Im(B)$ of the associated mutation fan.
\end{proof}

Under certain conditions, the relation between the Zamolodchikov fan and the mutation fan results in the following statement first conjectured in \cite{chin}.
\begin{corollary}
\label{chamber count}
    Let $\Lambda$ be a finite type cluster algebra of rank $n$, and let $B$ be an (extended) exchange matrix associated to $\Lambda$.
    When $B$ is full rank, the number of cones in $\mathcal{Z}_B$ is $c_{\Lambda}$, the number of seeds in the cluster algebra of type $\Lambda$.
\end{corollary}
\begin{proof}
    When $B$ is full rank, $\Im(B^{\top}) = \mathbb{R}^n$, so by Proposition \ref{mutation fan} $\mathcal{Z}_B$ is combinatorially equivalent to the mutation fan.
    In finite type, the mutation fan coincides with the $g$-vector fan of $B^{\top}$, whose cones are in correspondence with the seeds in the cluster algebra of type $\Lambda$ \cite{reading, fomin4}.
\end{proof}

\begin{remark}
    Although we focus on finite type cluster algebras, these results weakly extend to all of the Zamolodchikov periodic families.
    As described in the proof of Lemma \ref{finite type fans}, $\mathcal{Z}$ is always equivalent to the restricted cluster variable fan generated by all cluster variables in the bipartite belt.
    Beyond finite type, this cannot be directly related to the mutation fan, but still can be related to the Zamolodchikov mutation fan.
    By the same argument outlined in the proof of Proposition \ref{mutation fan}, $\mathcal{Z}_B$ is combinatorially equivalent to the subspace $\Im(B)$ of the Zamolodchikov mutation fan.
    When $B$ is full rank, $\mathcal{Z}_B$ is combinatorially equivalent to the Zamolodchikov mutation fan.
    Figure \ref{fig:fan table} depicts the relationship between all of the defined fans in this section.
\end{remark}

\section{Type $A_n$ laminations}
\label{sec:A}

In this section, we use type $A$ laminations to prove Theorem \ref{mainThm} for $A_n$ and $C_n$, as well as Theorem \ref{thm:ZamTropGr}.
We first focus on laminations of type $A$ triangulations, then eventually move on to introducing type $C$ laminations through folding.

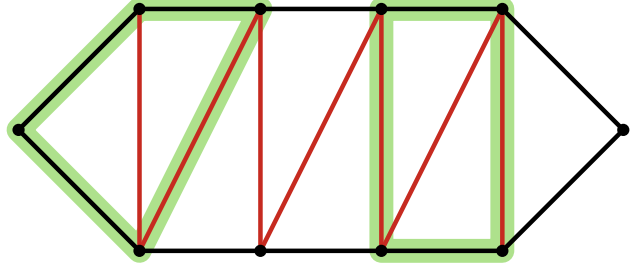
\begin{figure}
\centering
\definecolor{zzred}{RGB}{198,42,32}
\definecolor{zzhl}{RGB}{160,220,120}
\begin{tikzpicture}[
    line cap=round, line join=round,
    edge/.style={line width=1.7pt, black},
    diag/.style={line width=1.7pt, zzred},
    hl/.style={line width=9pt, zzhl, opacity=0.85},
    vtx/.style={circle, fill=black, inner sep=0pt, minimum size=4.6pt}
  ]

  \def\dx{1.6}    
  \def\dy{3.2}    
  \def\ap{1.6}    

  \pgfmathsetmacro{\ymid}{\dy/2}

  \coordinate (L)  at (-\ap,\ymid);          
  \coordinate (T1) at (0,\dy);
  \coordinate (T2) at (\dx,\dy);
  \coordinate (T3) at (2*\dx,\dy);
  \coordinate (T4) at (3*\dx,\dy);
  \coordinate (R)  at (3*\dx+\ap,\ymid);     
  \coordinate (B4) at (3*\dx,0);
  \coordinate (B3) at (2*\dx,0);
  \coordinate (B2) at (\dx,0);
  \coordinate (B1) at (0,0);

  \draw[hl] (L) -- (T1) -- (T2) -- (B1) -- cycle;   
  \draw[hl] (T3) -- (T4) -- (B4) -- (B3) -- cycle;  

  \draw[edge] (L) -- (T1) -- (T2) -- (T3) -- (T4) -- (R)
                  -- (B4) -- (B3) -- (B2) -- (B1) -- cycle;

  \draw[diag] (T1) -- (B1);
  \draw[diag] (T2) -- (B1);
  \draw[diag] (T2) -- (B2);
  \draw[diag] (T3) -- (B2);
  \draw[diag] (T3) -- (B3);
  \draw[diag] (T4) -- (B3);
  \draw[diag] (T4) -- (B4);

  \foreach \v in {L,T1,T2,T3,T4,R,B4,B3,B2,B1}{\node[vtx] at (\v) {};}

\end{tikzpicture}
\caption{\label{fig:A triangulations} A typical type $A_7$ ``zigzag" triangulation that appears in the bipartite belt.  Highlighted in green are two different quadrilaterals representing mutations that appear in the bipartite belt.}
\end{figure}

\subsection{Homomesy in type $A$ tropical $T$-systems}
We begin by visualizing the progression of $A_n$ through the bipartite belt as triangulations of an $(n+3)$-gon.
Every seed appears as a ``zigzag" triangulation, where every quadrilateral has at least two sides on the boundary of the polygon (see Figure \ref{fig:A triangulations}).
Each time $\mu_{\circ}$ or $\mu_{\bullet}$ is applied, half of the diagonals of the triangulation are flipped.
One application of $\mu_{\circ}\mu_{\bullet}$ rotates the entire triangulation by $\frac{2\pi}{n+3}$, so $n+3$ applications result in the original triangulation (i.e. the original seed).
Thus, the length of the period for type $A_n$ is $2(n+3)$.

We first check that type $A$ laminations work well with cluster mutation.

\begin{lemma}
\label{lem:A laminations}
    Let $L$ be a weighted lamination of an $(n+3)$-gon, and let $(\tilde{\boldsymbol{x}}, \tilde{B})$ be a seed of type $A_n$ with triangulation coefficients.
    Let $\lambda\in\mathbb{R}^m$ be the tropical labeling determined by $L$ as outlined in \ref{eq:lamination}.
    Then, applying a mutation $\mu_k$ results in $\lambda_k'$, which also satisfies \ref{eq:lamination} with respect to the new triangulation associated to $\mu_k(\tilde{\boldsymbol{x}}, \tilde{B})$.
\end{lemma}

\begin{proof}
    We must check that the exchange relations (\ref{eq:tropicalMutation}) hold for any lamination when flipping a diagonal in a generic mutation quadrilateral.
    Even though the mutation quadrilateral may not lie in the bipartite belt, for the purposes of this proof it is helpful to distinguish the pairs of sides by coloring them red or blue (see Figure \ref{fig:quad types}).
    
    Mutation has the form
    \begin{equation}
    \label{eq:mutation quad 1}
        \lambda_k + \lambda_k' = \max(\textcolor{red}{\lambda_1 + \lambda_2}, \textcolor{blue}{\lambda_3 + \lambda_4}).
    \end{equation}
    We check that each lamination curve preserves the above exchange relation.
    Each of the neutral lamination curves (elbows) intersects two consecutive sides and one of the crossing diagonals.
    So, this lamination contributes once to the red sum, once to the blue sum, and once to $\lambda_k + \lambda_k'$, preserving (\ref{eq:mutation quad 1}).
    These neutral lamination curves are compatible with any other choice of lamination curves.
    
    A blue lamination curve passes through both crossing diagonals and both blue sides.
    This contributes twice to the blue sum and twice to $\lambda_k + \lambda_k'$.
    A red lamination curve is incompatible with a blue lamination curve, and passes through both crossing diagonals and both red sides.
    This contributes twice to the red sum and twice to $\lambda_k + \lambda_k'$.
    Either case preserves (\ref{eq:mutation quad 1})
\end{proof}

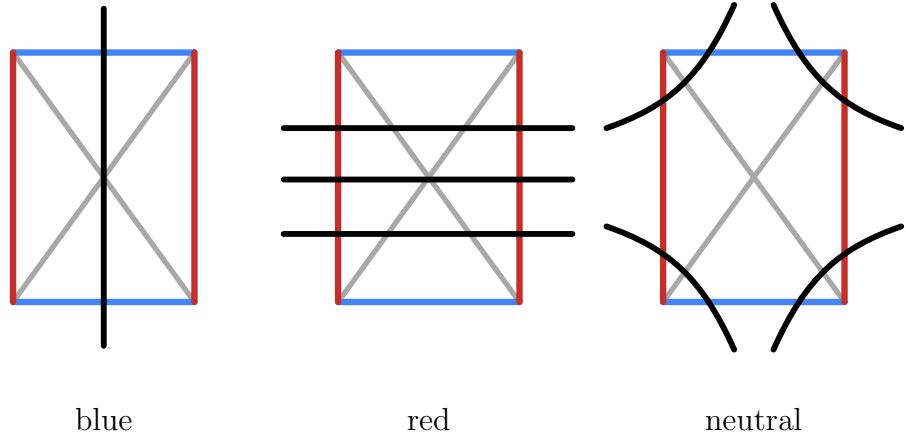
\begin{figure}
\centering
\definecolor{qblue}{RGB}{66,133,244}
\definecolor{qred}{RGB}{200,46,40}
\definecolor{qgrey}{RGB}{168,168,168}
\begin{tikzpicture}[
    line cap=round, line join=round,
    bl/.style={line width=2.4pt, qblue},
    rd/.style={line width=2.4pt, qred},
    gy/.style={line width=2pt, qgrey},
    lam/.style={line width=2.2pt, black},
    lbl/.style={font=\normalsize, anchor=north}
  ]

  \def\W{2.4}    
  \def\H{3.3}    
  \def\P{4.3}    
  \def\LY{-1.25} 
  \def\quadshape{%
    \draw[gy] (0,0) -- (\W,\H);   \draw[gy] (0,\H) -- (\W,0);
    \draw[bl] (0,\H) -- (\W,\H);  \draw[bl] (0,0) -- (\W,0);
    \draw[rd] (0,0) -- (0,\H);    \draw[rd] (\W,0) -- (\W,\H);}

  \begin{scope}[shift={(0,0)}]
    \quadshape
    \draw[lam] (\W/2,-0.58) -- (\W/2,\H+0.58);
    \node[lbl] at (\W/2,\LY) {blue};
  \end{scope}

  \begin{scope}[shift={(\P,0)}]
    \quadshape
    \foreach \y in {0.90,1.62,2.30}{\draw[lam] (-0.72,\y) -- (\W+0.70,\y);}
    \node[lbl] at (\W/2,\LY) {red};
  \end{scope}

  \begin{scope}[shift={(2*\P,0)}]
    \quadshape
    \foreach \sx/\sy in {1/1,-1/1,1/-1,-1/-1}{
      \begin{scope}[shift={(\W/2,\H/2)}, xscale=\sx, yscale=\sy]
        \draw[lam] (-\W/2-0.75,\H/2-1.00)
          .. controls (-\W/2+0.05,\H/2-0.72) and (-\W/2+0.50,\H/2-0.35) ..
             (-\W/2+0.94,\H/2+0.63);
      \end{scope}}
    \node[lbl] at (\W/2,\LY) {neutral};
  \end{scope}

\end{tikzpicture}
\caption{\label{fig:quad types} Three types of lamination curves that can pass through a mutation quadrilateral in type $A_n$.  The left contains a single curve that contributes only to the {\color{qblue}$\Delta$} portion of the mutation.  The middle contains several curves that contribute only to the {\color{qred}$\Gamma$} portion of the mutation.  On the right, all lamination curves contribute equally to both {\color{qblue}$\Delta$} and {\color{qred}$\Gamma$} sums.}
\end{figure}

The following lemma states that tropical labelings for type $A$ triangulations are invariant of shifting by a constant.
\begin{lemma}
\label{Type A: shifting labelings}
    Let $B$ be the extended exchange matrix associated to the bipartite orientation of $A_n$ with triangulation coefficients.
    Let $\lambda\in\mathbb{R}^{2n+3}$ be a tropical labeling of $B$, and let $c\in\mathbb{R}$ be a constant scalar.
    Define a new labeling $\rho\in\mathbb{R}^{2n+3}$ given by $\rho_i = \lambda_i + c$, $\forall i\in [2n+3]$.
    Then, $\lambda$ and $\rho$ generate the same colored mutation sequence in the bipartite belt, and therefore lie in the same cone of $\mathcal{Z}_{2,n}$.
\end{lemma}
\begin{proof}
    Recall the form for tropical mutation (\ref{eq:tropicalMutation}). 
    Every cluster variable is a diagonal that borders two triangles, with two incoming arrows and two outgoing arrows.
    Mutation at $k$ results in
    \begin{align*}
        \boldsymbol{t}_k^{\rho}(t+1) &= \max\Big(\big(\boldsymbol{t}_{i_1}^{\lambda}(t) + c\big) + \big(\boldsymbol{t}_{i_2}^{\lambda}(t) + c\big), \big(\boldsymbol{t}_{j_1}^{\lambda}(t) + c\big) + \big(\boldsymbol{t}_{j_2}^{\lambda}(t) + c\big)\Big) - \big(\boldsymbol{t}_k^{\lambda}(t-1) + c\big)\\
        &= \max\left(\boldsymbol{t}_{i_1}^{\lambda}(t) + \boldsymbol{t}_{i_2}^{\lambda}(t), \boldsymbol{t}_{j_1}^{\lambda}(t) + \boldsymbol{t}_{j_2}^{\lambda}(t)\right) + 2c - \boldsymbol{t}_k^{\lambda}(t-1) - c\\
        &= \boldsymbol{t}_k^{\lambda}(t+1) + c,
    \end{align*}
    where $i_1, i_2$ are the sources of incoming arrows, and $j_1, j_2$ are the targets of outgoing arrows.
    
    So, $\boldsymbol{t}^{\rho}(t) = \boldsymbol{t}^{\lambda}(t) + c$ for all $t\in \mathbb{Z}$, and these two tropical $T$-systems have equivalent exchange relations in the bipartite belt.
    Thus a mutation in $\boldsymbol{t}^{\lambda}$ is blue (resp., red) if and only if it is blue (resp., red) in $\boldsymbol{t}^{\rho}$.
\end{proof}

\begin{lemma}
\label{lem:bij}
The Zamolodchikov fan $\mathcal Z_{2,n}$ is the fan of weighted laminations
of the $n$-gon.  Let $L$ be a maximal lamination, and write
\[
L=L^\circ\sqcup S,
\]
where $L^\circ$ is the set of nontrivial curves and $S$ is the set of
side-like curves.  Then the corresponding cone is
\[
\sigma_L
=
\operatorname{span}_{\mathbb R}\{v_s:s\in S\}
+
\sum_{\ell\in L^\circ}\mathbb R_{\ge0}v_\ell,
\]
where $v_\ell$ is the tropical vector~(\ref{eq:lamination}) associated with
the singleton lamination $(\ell)$.

Every cone of $\mathcal Z_{2,n}$ is obtained in this way.  Thus the
side-like curves span the common lineality space, while the rays of the
pointed fan are indexed by single nontrivial lamination curves, and the
chambers by maximal laminations.
\end{lemma}

\begin{proof}
    Every mutation in the bipartite belt involves flipping a diagonal within a mutation quadrilateral that has two opposite sides lying on the boundary of the polygon.
    With (\ref{eq:tropicalMutation}) in mind, this pair of boundary sides sums to the $\Delta$ portion of the tropical mutation, while the other pair of sides sums to the $\Gamma$ portion of the mutation.
    The mutation is blue if the $\Delta$-sum is greater than the $\Gamma$-sum, and red if the $\Gamma$-sum is greater than the $\Delta$-sum.

    Consider the tropical vector $v_\ell$ corresponding to a single lamination curve $\ell$ of weight $1$.
    Then $v_\ell$ assigns $1$ to every diagonal intersected by $\ell$, and $0$ to every other diagonal; see Figure~\ref{fig:lamination} and (\ref{eq:lamination}).  For a given mutation quadrilateral, the curve $\ell$ may determine the mutation to be blue, red, or neither.
    Notice for a fixed quadrilateral, there is only one lamination curve that can make the mutation blue, whereas there are potentially many lamination curves that can make the mutation red (see Figure \ref{fig:quad types}).
    By Lemma \ref{Type A: shifting labelings}, it suffices to only consider positive tropical labelings.
    Given a lamination, we can determine the color of each mutation by examining its lamination curves and how they intersect with each quadrilateral.
    A blue lamination curve is incompatible with any red lamination curve within a quadrilateral, so the placement of either curve completely determines the color of that mutation.

    In fact, the colored mutation sequences are in natural bijection with maximal laminations.
    Given a maximal lamination, we can uniquely identify the color of each mutation. 
    Consider a quadrilateral that corresponds to a mutation in the bipartite belt.
    If the color is undetermined, then there is neither a blue lamination curve nor any red lamination curves.
    Thus one can add a blue lamination curve to that quadrilateral, and it will be compatible with the rest of the lamination, contradicting the maximality of the original lamination.
    For uniqueness, let $L$ and $L'$ be two distinct maximal laminations, and let $\ell\in L$  be a lamination curve in $L$ that is not in $L'$.
    There is a unique mutation quadrilateral containing $\ell$ which will result in a blue mutation for $L$, but a red mutation for $L'$.
    So, distinct maximal laminations will produce distinct color sequences.
    Finally, we can check these maximal laminations truly cover all cones in the Zamolodchikov fan by counting the number of maximal laminations.
    This is equivalent to counting the number of triangulations of an $(n+3)$-gon, which is the Catalan number $C_{n+1} = \frac{1}{n+2}\cdot {2n+2\choose n+1}$ and also the number of seeds in the cluster algebra of type $A_n$.
    By Corollary \ref{chamber count}, this covers every chamber in the Zamolodchikov fan.

    Given a color sequence, one can easily uniquely construct a maximal lamination.
    Take any two nonadjacent sides of the polygon. 
    Then there will be a nontrivial lamination curve connecting these two sides exactly when the quadrilateral generated by those sides forms a blue mutation.

    It remains to check that these generators span the corresponding chamber.  If $L=\{\ell_1,\ldots,\ell_r\}$, then assigning nonnegative weights $a_1,\ldots,a_r$ to its curves gives the vector $a_1v_{\ell_1}+\cdots+a_rv_{\ell_r}$ in the tropical seed coordinates.  Changing the positive weights does not change which lamination curves determine the color in each quadrilateral, and therefore does not change the color sequence.  

\begin{figure}
\centering
\definecolor{lamblue}{RGB}{31,78,148}
\begin{tikzpicture}[
     scale=0.85, every node/.append style={transform shape},
    edge/.style={line width=1.9pt, black},
    arc/.style={line width=1.9pt, lamblue, line cap=round},
    wt/.style={lamblue, font=\small, inner sep=1pt}
  ]

  \def\R{3.2}    
  \def\vs{0.80}  

  \pgfmathsetmacro{\hy}{0.8660254*\R*\vs}   
  \coordinate (A) at (-0.5*\R, \hy);   
  \coordinate (B) at ( 0.5*\R, \hy);   
  \coordinate (C) at ( \R    , 0);     
  \coordinate (D) at ( 0.5*\R,-\hy);   
  \coordinate (E) at (-0.5*\R,-\hy);   
  \coordinate (F) at (-\R    , 0);     

  \draw[edge] (A) -- (B) -- (C) -- (D) -- (E) -- (F) -- cycle;
  \draw[edge] (B) -- (F);
  \draw[edge] (B) -- (E);
  \draw[edge] (E) -- (C);

  \draw[arc] (-3.00,1.78) .. controls (-1.10,0.05) and (1.85,0.05) .. (2.85,1.75);

  \node[wt] at ( 0.00, 2.60) {$0$};   
  \node[wt] at (-2.25, 1.8) {$1$};   
  \node[wt] at ( 2.8, 1.1) {$1$};   
  \node[wt] at (-2.78,-1.30) {$0$};   
  \node[wt] at ( 0.00,-2.60) {$0$};   
  \node[wt] at ( 2.78,-1.30) {$0$};   
  \node[wt] at (-0.85, 1.45) {$1$};   
  \node[wt] at ( 0.28,-0.1) {$1$};   
  \node[wt] at ( 1.00,-1.32) {$0$};   

\end{tikzpicture}
\caption{\label{fig:lamination} A single lamination curve of weight $1$ contributing to four diagonals: two cluster variables and two frozen variables.}
\end{figure}
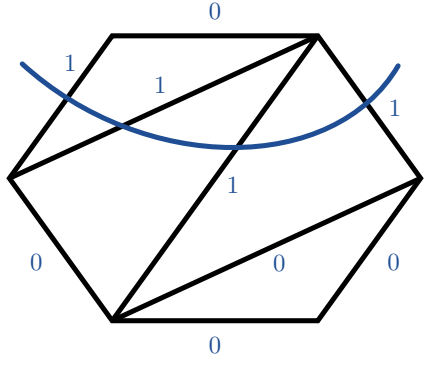

    Conversely, let $x$ be a positive labeling with the color sequence corresponding to $L$.  For each nontrivial curve $\ell_i\in L$, consider the unique mutation quadrilateral in which $\ell_i$ gives the blue contribution, and let $a_i$ be the difference between the blue and red sides there.  Since the other curves of $L$ are compatible with $\ell_i$, they contribute equally to the two sides, while $v_{\ell_i}$ contributes $1$ to their difference.  Thus these differences read off the weights of the nontrivial curves.

    Subtract $\sum_i a_i v_{\ell_i}$ from $x$, and call the resulting vector $y$. The blue and red sides are now equal in each of these quadrilaterals. The corresponding differences are linearly independent, and there is one for each mutable direction. Their common kernel therefore has dimension equal to the number of coefficient directions. The vectors of the side-like curves lie in this kernel and are linearly independent, so they span it. Hence $y$ is a linear combination of the side-like lamination vectors. We conclude that
    \[
    x\in \operatorname{span}_{\mathbb R}\{v_s:s\in S\}+\sum_{\ell\in L^\circ}\mathbb R_{\ge0}v_\ell.
    \]
    We already showed that every vector in this cone has the color sequence corresponding to $L$, so this is exactly the chamber corresponding to $L$.

    Allowing some of the nontrivial weights to be zero gives all faces of the chamber. In particular, after quotienting by the common lineality space, the rays are generated by the vectors $v_\ell$ corresponding to single nontrivial lamination curves.

\end{proof}

Now we are ready to prove Theorem \ref{mainThm} for type $A$.
\begin{proposition}
\label{A}
    Let $B$ be the extended exchange matrix associated to the bipartite orientation of $A_n$ with triangulation coefficients.
    In one period, there are $n(n+1)$ red mutations and $2n$ blue mutations.
\end{proposition}

\begin{proof}

By \cite{chin2}, the behavior of the first half-period is equivalent to the behavior of the second half-period up to an automorphism of order two.
Together with \cite{fomin2}, this shows each cluster variable in finite type appears exactly once in each half-period.
Every mutation quadrilateral will then appear once in each half-period.
By Lemma~\ref{lem:bij}, each chamber of $\mathcal Z_{2,n}$ corresponds to a maximal lamination.
From the construction above, it is clear the number of blue mutations in one half-period is $n$, the number of nontrivial curves in the maximal lamination.
So in one full period, there are $2n$ blue mutations, and the remaining $n(n+1)$ mutations are red.
\end{proof}

From the methodology used to prove Proposition \ref{A}, we can easily resolve Theorem \ref{thm:ZamTropGr}.
\begin{proof}[Proof of Theorem \ref{thm:ZamTropGr}]
By Lemma~\ref{lem:bij}, the Zamolodchikov fan for $G(2,n)$ is the fan of
weighted laminations of the $n$-gon.  On the other hand,
$\operatorname{Trop}^+G(2,n)$ is the fan of planar metric trees
\cite{SpeyerWilliams}.  A planar tree is dual to a lamination of the polygon,
with internal edges corresponding to lamination curves and carrying the same
weights.  Thus the two fans coincide.
\end{proof}

\begin{remark} 
\label{rem:diff}
    Already for $G(3,6)$ the two fans are different. Indeed, our computation gives
    \[
    f(\mathcal Z_{3,6})=(16,66,96,46),
    \]
    while Speyer and Williams computed
    \[
    f\bigl(\operatorname{Trop}^+G(3,6)\bigr)=(16,66,98,48)
    \]
    in \cite{SpeyerWilliams}. Thus the two fans have the same rays and the same
    $2$-dimensional cones, but differ in higher dimensions. 
\end{remark}

\begin{remark}
    The original conjecture was stated for Dynkin diagrams with trivial coefficients. 
    This can easily be recovered from Proposition \ref{A} by setting all frozen labelings to zero.
    As mentioned in Section \ref{sec:fans}, this may result in a not full rank $B$-matrix, and therefore fewer cones in $\mathcal{Z}$ than seeds in the cluster algebra.
    Due to Lemma \ref{Type A: shifting labelings}, if $B$ has no coefficients, cones of $\mathcal{Z}_B$ can be visualized as maximal laminations with a positive solution to all sides of the polygon being equal weight.
    For odd $n$, not all maximal laminations will have such a solution, which eliminates some cones.
\end{remark}

\begin{example}
    For $A_3$, there should be 14 cones, which corresponds to 14 triangulations of a hexagon and 14 maximal laminations. 
    However, eight of these laminations do not have positive weight solutions to all frozens being equal.
    So without any frozens, $\mathcal{Z}_B$ only has 6 cones.
    Figure \ref{fig:maximal laminations} depicts two examples of maximal laminations with no positive solutions, and one maximal lamination with a positive solution.
\end{example}

\begin{figure}[htbp]
\centering
\definecolor{lamred}{RGB}{196,45,40}
\definecolor{lamgreen}{RGB}{28,118,90}
\begin{tikzpicture}[
    line cap=round, line join=round,
    hex/.style={line width=1.7pt, black},
    lamr/.style={line width=1.6pt, lamred},
    lamg/.style={line width=1.6pt, lamgreen},
    wt/.style={font=\small, inner sep=1pt}
  ]
 
  \def\R{1.7}     
  \def\vs{0.80}   
  \def\Px{4.85}   
 
  \def\hexpath{(0:\R) -- (60:\R) -- (120:\R) -- (180:\R) -- (240:\R) -- (300:\R) -- cycle}
 
  \def\cornerarc{(1.214*\R,0.309*\R)
      .. controls (0.742*\R,0.15*\R) and (0.742*\R,-0.15*\R) ..
         (1.214*\R,-0.309*\R)}
  \def\onesidearc{(-0.979*\R,0.716*\R)
      .. controls (-0.50*\R,0.495*\R) and (0.50*\R,0.495*\R) ..
         (0.979*\R,0.716*\R)}
  \def\twosidearc{(-0.02*\R,1.206*\R) -- (-0.02*\R,-1.206*\R)}
 
  \begin{scope}[shift={(0,0)}, yscale=\vs]
    \draw[hex] \hexpath;
    \foreach \a in {0,60,...,300}{\draw[lamr,rotate=\a] \cornerarc;}
    \foreach \a in {0,120,240}{\draw[lamr,rotate=\a] \onesidearc;}
  \end{scope}
 
  \begin{scope}[shift={(\Px,0)}, yscale=\vs]
    \draw[hex] \hexpath;
    \foreach \a in {0,60,...,300}{\draw[lamr,rotate=\a] \cornerarc;}
    \draw[lamr,rotate=0]   \onesidearc;
    \draw[lamr,rotate=60]  \twosidearc;
    \draw[lamr,rotate=120] \onesidearc;
  \end{scope}
 
  \begin{scope}[shift={(2*\Px,0)}, yscale=\vs]
    \draw[hex] \hexpath;
    \foreach \a in {0,60,...,300}{\draw[lamg,rotate=\a] \cornerarc;}
    \draw[lamg,rotate=-60] \onesidearc;
    \draw[lamg,rotate=60]  \twosidearc;
    \draw[lamg,rotate=120] \onesidearc;
    \foreach \x/\y/\w in {%
      -0.339/1.37/1,  0.11/1.37/1,  0.38/1.37/3,
       1.014/0.977/3, 1.353/0.389/2,
       1.353/-0.30/2, 1.269/-0.57/1, 1.15/-0.800/1, 0.95/-1.1/1,%
       0.36/-1.37/1, -0.11/-1.37/1, -0.38/-1.37/3,
      -1.014/-0.977/3, -1.353/-0.389/2,
      -1.353/0.30/2, -1.269/0.57/1, -1.15/0.800/1, -0.95/1.1/1}%
      {\node[wt] at (\x*\R,\y*\R) {$\w$};}
  \end{scope}
 
\end{tikzpicture}
\caption{\label{fig:maximal laminations} Three maximal laminations for $A_3$.  The right lamination is labeled with positive weights ensuring all sides have weight $5$.}
\end{figure}
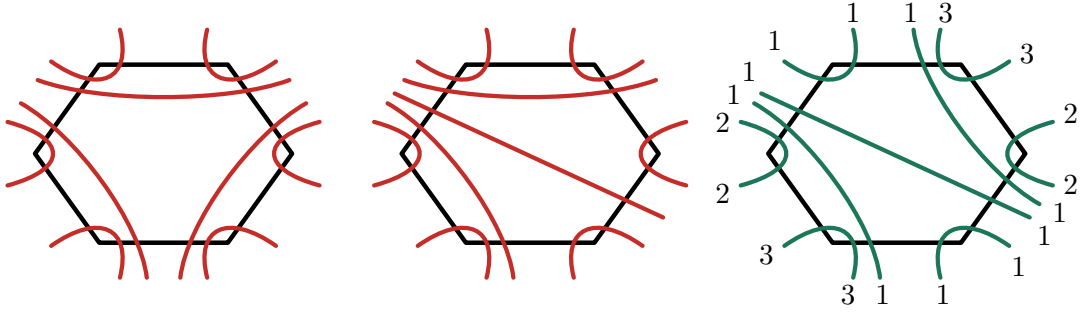

\subsection{Homomesy in type $C$ tropical $T$-systems}

Type $C_n$ follows easily from type $A_{2n-1}$ using a standard folding argument.
While type $A$ laminations consist of any lamination curves on a polygon, type $C$ laminations restrict this set to only $180^{\circ}$ rotationally symmetric laminations.

As defined in \cite{chin}, a \textit{bicolored automorphism} $f: [n] \rightarrow [n]$ is an automorphism that satisfies the following.
\begin{itemize}
    \item[(i)]
        $\epsilon_{f(i)} = \epsilon_i$ (preserves bipartite coloring of vertices).
    \item[(ii)]
        For all $i_1, i_2$ in the same $f$-orbit $I$, and for all $j\in[n]$, $b_{i_1j}b_{i_2j} \geq 0$ (preserves edge colors).
    \item[(iii)]
        For all $i, j\in [n]$, $b_{f(i)f(j)} = b_{ij}$.
    \item[(iv)]
        For all $i, j$ in the same $f$-orbit $I$, $b_{ij} = b_{ji} = 0$.
\end{itemize}

\begin{corollary}
    Let $B$ be the full rank exchange matrix associated to the bipartite orientation of $C_n$ with triangulation coefficients. In one period, there are $2n^2$ red mutations and $2n$ blue mutations.
\end{corollary}

\begin{figure}
\centering
\definecolor{foldred}{RGB}{196,45,40}
\definecolor{lamblue}{RGB}{31,78,148}
\begin{tikzpicture}[
    line cap=round, line join=round,
    oct/.style={line width=1.9pt, black},
    diag/.style={line width=2pt, foldred},
    lam/.style={line width=1.9pt, lamblue},
    lbl/.style={font=\small, inner sep=0pt}
  ]

  \def\R{2.2}     
  \def\Px{5.94}   

  \def\octpath{(22.5:\R) -- (67.5:\R) -- (112.5:\R) -- (157.5:\R)
            -- (202.5:\R) -- (247.5:\R) -- (292.5:\R) -- (337.5:\R) -- cycle}

  \begin{scope}[shift={(0,0)}]
    \draw[oct] \octpath;
    \foreach \k in {0,...,7}{\coordinate (V\k) at ({22.5+45*\k}:\R);}
    \draw[diag] (V3) -- (V5);
    \draw[diag] (V2) -- (V5);
    \draw[diag] (V2) -- (V6);
    \draw[diag] (V1) -- (V6);
    \draw[diag] (V1) -- (V7);
    \def\markpt#1#2#3#4{\fill ($(#1)!#3!(#2)$) circle[radius=2.4pt];
        \node[lbl,right=3pt] at ($(#1)!#3!(#2)$) {$#4$};}
    \markpt{V3}{V5}{0.47}{\scalebox{0.85}{$x_3$}}
    \markpt{V2}{V5}{0.57}{\scalebox{0.85}{$x_2$}}
    \markpt{V2}{V6}{0.50}{\scalebox{0.85}{$x_1$}}
    \markpt{V1}{V6}{0.45}{\scalebox{0.85}{$x_{2'}$}}
    \markpt{V1}{V7}{0.535}{\scalebox{0.85}{$x_{3'}$}}
  \end{scope}

  \begin{scope}[shift={(\Px,0)}]
    \draw[oct] \octpath;
    \def\cornerarc{(1.200*\R,0.248*\R)
        .. controls (0.773*\R,0.12*\R) and (0.773*\R,-0.12*\R) ..
           (1.200*\R,-0.248*\R)}
    \foreach \a in {22.5,67.5,...,337.5}{\draw[lam,rotate=\a] \cornerarc;}
    \draw[lam] (-0.150*\R,1.202*\R) -- (0.150*\R,-1.202*\R);
    \foreach \s in {0,180}{
      \draw[lam,rotate=\s] (0.115*\R,1.204*\R)
          .. controls (0.115*\R,0.45*\R) and (0.55*\R,0) .. (1.204*\R,0);
      \draw[lam,rotate=\s] (0,1.204*\R)
          .. controls (0,0.40*\R) and (0.501*\R,-0.501*\R) ..
             (0.851*\R,-0.851*\R);}
  \end{scope}

\end{tikzpicture}
\caption{\label{fig:folding} The left depicts folding $A_5$ into $C_3$.  The orbits of $f$ are $\{x_1\}, \{x_2,x_{2'}\}$, and $\{x_3,x_{3'}\}$.  On the right is an example of a rotationally symmetric maximal lamination.}
\end{figure}
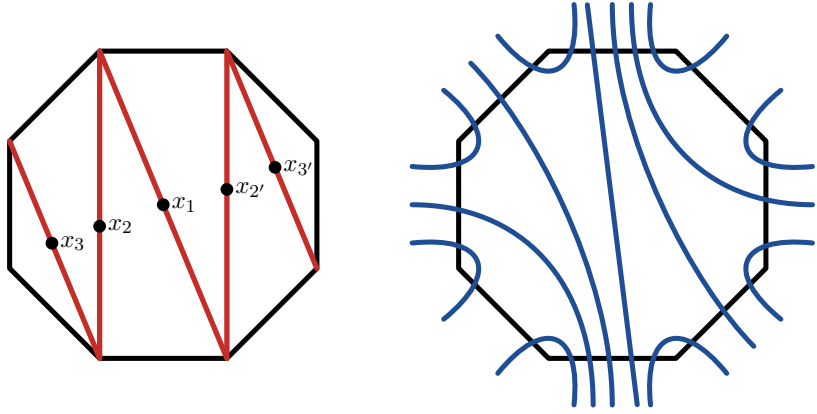

\begin{proof}
    Let $f$ be the bicolored automorphism that folds $A_{2n-1}$ into $C_n$.
    Maximal laminations are foldable into type $C$ if they are rotationally symmetric (including weights).
    Let $f$ be the bicolored automorphism that folds $A_{2n-1}$ into $C_n$ (see Figure \ref{fig:folding}).
    Notice that a rotationally symmetric lamination ensures that any two cluster variables in the same $f$-orbit have the same value at any given time.
    We may then view mutating two variables in the same $f$-orbit as one combined mutation.

    Now count the number of blue mutations.
    A rotationally symmetric maximal lamination consists of $2n-1$ curves: one curve bisecting the polygon, and $n-1$ curves within each resulting half of the polygon.
    The laminations in one half of the polygon determine the entire maximal lamination, so we look only at one half.
    The curve that bisects the polygon is blue for exactly one quadrilateral, which corresponds to the center cluster variable in type $A$.
    This contributes one blue mutation.
    Any other curve $\ell$ has a counterpart $\ell'$ in the other half of the polygon. 
    These form two blue quadrilaterals that are rotationally symmetric, and correspond to mutating at two cluster variables in the same $f$-orbit.
    So, this contributes one blue mutation.
    This, the total number of blue mutations in one half period is $n$.

    Notice that the number of rotationally symmetric maximal laminations is equivalent to choosing a bisecting curve, then counting the number of triangulations of an $(n+2)$-gon:
    $$(n+1)\cdot C_n = (n+1)\cdot\frac{1}{n+1}{2n\choose n} = {2n\choose n}.$$
    This is the number of seeds of a cluster algebra of type $C_n$, so by Corollary \ref{chamber count}, this covers the full Zamolodchikov fan for type $C_n$.
\end{proof}

\section{Type $D_n$ laminations} \label{sec:D}
In this section, we use type $D$ laminations to prove Theorem \ref{mainThm} for $D_n$ and $B_n$.

\subsection{Homomesy in type $D$ tropical $T$-systems}
Before using these laminations, we must first check that they work well with cluster mutation in the bipartite belt.

Let us begin by visualizing the progression of $D_n$ through the bipartite belt as tagged triangulations of an $n$-gon.
Every seed appears as a ``zigzag" tagged triangulation which is slightly different from the usual type $A$ triangulation (see Figure \ref{fig:D triangulations}).
One application of $\mu_{\circ}\mu_{\bullet}$ rotates the entire triangulation by $\frac{2\pi}{n}$, so $n$ applications result in the original tagged triangulation.
Thus, the length of the period for type $D_n$ is $2n$.

\begin{figure}
\centering
\definecolor{zzred}{RGB}{198,42,32}
\begin{tikzpicture}[
    x=1.45cm, y=1.45cm,
    line cap=round, line join=round,
    bdry/.style={line width=2.1pt, black},
    arc/.style={line width=2pt, zzred},
    dot/.style={circle, fill=black, inner sep=0pt, minimum size=4.8pt}
  ]

  \coordinate (L)  at (-1,1);            
  \foreach \k in {0,1,2,3}{
    \coordinate (T\k) at (\k,2);         
    \coordinate (B\k) at (\k,0);}        
  \coordinate (P) at (3.42,1);           

  \draw[bdry] (L) -- (T0) -- (T1) -- (T2) -- (T3)
              arc[start angle=90, end angle=-90, radius=1]
              -- (B2) -- (B1) -- (B0) -- cycle;

  \draw[arc] (T0) -- (B0);
  \draw[arc] (T1) -- (B0);
  \draw[arc] (T1) -- (B1);
  \draw[arc] (T2) -- (B1);
  \draw[arc] (T2) -- (B2);
  \draw[arc] (T3) -- (B2);
  \draw[arc] (T3) -- (B3);

  \draw[arc] (B3) .. controls (3.259,0.48) and (3.259,0.48) .. (P);
  \draw[arc] (B3) .. controls (3.579,0.345) and (3.579,0.345) .. (P)
      coordinate[pos=0.55] (nk);
  \begin{scope}[shift={(nk)}, rotate=-14]   
    \draw[zzred, line width=1.3pt]
        (0,0) -- (-0.10, 0.075) -- (-0.10,-0.075) -- cycle;
    \draw[zzred, line width=1.3pt]
        (0,0) -- ( 0.10, 0.075) -- ( 0.10,-0.075) -- cycle;
  \end{scope}

  \foreach \v in {L,T0,T1,T2,T3,B0,B1,B2,B3,P}{\node[dot] at (\v) {};}

\end{tikzpicture}
\caption{\label{fig:D triangulations} A typical type $D_9$ ``zigzag" triangulation that appears in the bipartite belt.}
\end{figure}
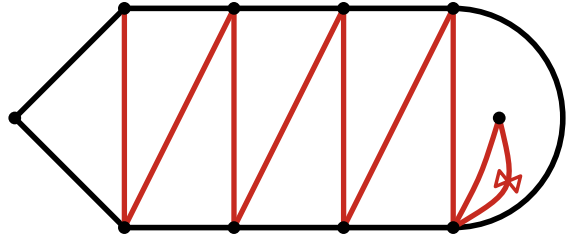

\begin{figure}
\centering
\definecolor{qblue}{RGB}{66,133,244}
\definecolor{qred}{RGB}{200,46,40}
\begin{tikzpicture}[
    line cap=round, line join=round,
    bl/.style={line width=2.2pt, qblue},
    rd/.style={line width=2.2pt, qred},
    bk/.style={line width=2.2pt, black},
    lbl/.style={font=\small, inner sep=1pt},
    dot/.style={circle, fill=black, inner sep=0pt, minimum size=4.4pt}
  ]

  \def\Px{4.25}   
  \def\notchat#1#2{\begin{scope}[shift={(#1)}, rotate=#2]
      \draw[line width=1.3pt] (0,0) -- (-0.11, 0.08) -- (-0.11,-0.08) -- cycle;
      \draw[line width=1.3pt] (0,0) -- ( 0.11, 0.08) -- ( 0.11,-0.08) -- cycle;
    \end{scope}}

  \begin{scope}[shift={(0,0)}]
    \draw[rd] (-1.2,-1.71) -- (-1.2,1.71);      
    \draw[rd] ( 1.2,-1.71) -- ( 1.2,1.71);      
    \draw[bl] (-1.2, 1.71) -- ( 1.2,1.71);      
    \draw[bl] (-1.2,-1.71) -- ( 1.2,-1.71);     
    \draw[bk] (-1.2,-1.71) -- ( 1.2,1.71);      
    \node[lbl] at (-1.50,0) {$1$};
    \node[lbl] at ( 1.50,0) {$2$};
    \node[lbl] at (0, 2.05) {$3$};
    \node[lbl] at (0,-2.05) {$4$};
  \end{scope}

  \begin{scope}[shift={(\Px,0)}]
    \draw[rd] (0,1.715) .. controls (-2.133,0.40) and (-2.133,-0.40) .. (0,-1.715);
    \draw[bl] (0,1.715) .. controls ( 2.133,0.40) and ( 2.133,-0.40) .. (0,-1.715);
    \draw[bk] (0,0) .. controls (-0.493,-0.857) and (-0.493,-0.857) .. (0,-1.715);
    \draw[bk] (0,0) .. controls ( 0.493,-0.857) and ( 0.493,-0.857) .. (0,-1.715)
        coordinate[pos=0.38] (nkA);
    \notchat{nkA}{12}
    \node[dot] at (0,0) {};
    \node[lbl] at (-1.90,0) {$1$};
    \node[lbl] at ( 1.90,0) {$2$};
  \end{scope}

  \begin{scope}[shift={(2*\Px,0)}]
    \draw[rd] (-1.315,-1.71) -- (-1.315,1.71);                    
    \draw[bl] (-1.315, 1.71) .. controls (-0.20, 1.95) and (0.90, 1.15)
                             .. (1.315,0);                        
    \draw[bl] (-1.315,-1.71) .. controls (-0.20,-1.95) and (0.90,-1.15)
                             .. (1.315,0);                        
    \draw[rd] (-0.485,0) .. controls (0.415, 0.493) and (0.415, 0.493)
                         .. (1.315,0) coordinate[pos=0.42] (nkB); 
    \draw[rd] (-0.485,0) .. controls (0.415,-0.493) and (0.415,-0.493)
                         .. (1.315,0);                            
    \begin{scope}[qred]\notchat{nkB}{100}\end{scope}
    \draw[bk] (-1.315,-1.71) .. controls (-0.95,1.15) and (0.35,1.25) .. (1.315,0);
    \node[dot] at (-0.485,0) {};
    \node[lbl] at (-1.60, 0)    {$1$};
    \node[lbl] at (-0.065, 0.49){$2$};
    \node[lbl] at (-0.065,-0.46){$3$};
    \node[lbl] at ( 0.165, 1.9){$4$};
    \node[lbl] at ( 0.165,-1.9){$5$};
  \end{scope}

\end{tikzpicture}
\caption{\label{fig:D quadrilaterals} The three types of quadrilaterals that can appear in type $D_n$ tagged triangulations.}
\end{figure}
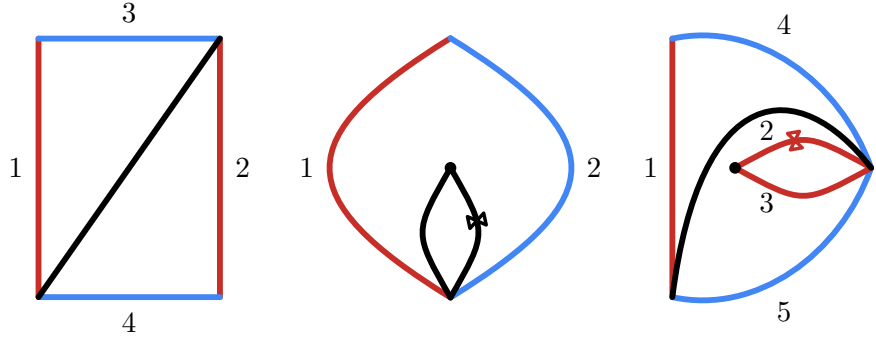

\begin{lemma} \label{lem:trop}
    Let $L$ be a weighted lamination of a punctured $n$-gon, and let $(\tilde{\boldsymbol{x}}, \tilde{B})$ be a seed of type $D_n$ in the bipartite belt with tagged triangulation coefficients.
    Let $\lambda\in \mathbb{R}^{m}$ be defined by the intersections of laminations as outlined in (\ref{eq:lamination}).
    Then, applying a mutation $\mu_k$ in the bipartite belt results in $\lambda_k'$, which also satisfies (\ref{eq:lamination}).
\end{lemma}
\begin{proof}
    Every mutation in the bipartite belt involves flipping a diagonal within a mutation quadrilateral, which comes in three forms, including the usual quadrilateral appearing in type $A$ (see Figure \ref{fig:D quadrilaterals}).
    We must check that the exchange relations (\ref{eq:tropicalMutation}) hold for any lamination of a mutation quadrilateral in the bipartite belt.

    \noindent
    \textbf{Case 1:} By Lemma \ref{lem:A laminations}, type $A$ mutation quadrilaterals preserve the exchange relations.

    \noindent
    \textbf{Case 2:} A bigon mutation quadrilateral is depicted in Figure \ref{fig:D bigon}.
    Mutation has the form
    \begin{align*}
        \lambda_k + \lambda_k' = \max(\textcolor{red}{\lambda_1}, \textcolor{blue}{\lambda_2}).
    \end{align*}
    We check that each lamination curve preserves the above exchange relation.
    Each of the neutral laminations intersects the two boundary sides and one of the ``diagonals", contributing once to the red sum, once to the blue sum, and once to $\lambda_k + \lambda_k'$.
    A blue lamination is a radius that corresponds to a blue mutation, and intersects the blue boundary side once.
    This contributes once to the blue sum, and once to $\lambda_k + \lambda_k'$.
    A red lamination is either a radius or a non-radius curve, and corresponds to a red mutation.
    A red radius intersects the red boundary side once, contributing once to the red sum and once to $\lambda_k + \lambda_k'$.
    A red non-radius curve intersects the red boundary side twice, and both of the ``diagonals", contributing twice to the red sum, and twice to $\lambda_k + \lambda_k'$.

    \noindent
    \textbf{Case 3:} The last mutation quadrilateral is depicted in Figure \ref{fig:D colored laminations}.
    Mutation has the form
    \begin{align*}
        \lambda_k + \lambda_k' = \max(\textcolor{red}{\lambda_1 + \lambda_2 + \lambda_3}, \textcolor{blue}{\lambda_4 + \lambda_5}).
    \end{align*}
    We check that each lamination curve preserves the above exchange relation.
    A neutral lamination (elbow) comes in two forms.
    Either it contributes once to the red sum, once to the blue sum, and once to the crossing diagonals $\lambda_k + \lambda_k'$, or it contributes twice to each sum.
    A blue lamination intersects each blue side and each crossing diagonal, contributing twice to the blue sum and twice to $\lambda_k + \lambda_k'$.
    A red lamination comes in two forms.
    A red lamination curve that does not intersect any blue sides contributes four times to the red sum and four times to the crossing diagonals $\lambda_k + \lambda_k'$.
    Otherwise, it intersects one red and one blue boundary side, as well as both radii, contributing three times to the red sum and three times to $\lambda_k + \lambda_k'$.
    There is also the possibility of a lamination radius, which acts neutrally when coming from a blue side, and is red otherwise.
    These also satisfy (\ref{eq:tropicalMutation}).
\end{proof}

\begin{figure}
\centering
\definecolor{bigred}{RGB}{196,42,36}
\definecolor{bigblue}{RGB}{66,133,244}
\definecolor{biggrey}{RGB}{168,168,168}
\begin{tikzpicture}[
    line cap=round, line join=round,
    rd/.style={line width=2.5pt, bigred},
    bl/.style={line width=2.5pt, bigblue},
    gy/.style={line width=2.5pt, biggrey},
    lam/.style={line width=2.2pt, black},
    dot/.style={circle, fill=black, inner sep=0pt, minimum size=5pt},
    lbl/.style={font=\normalsize, anchor=base}
  ]

  \def\H{1.685}   
  \def\Wc{1.667}  
  \def\Px{4.22}   
  \def\LY{-2.79}  
  \def\notchat#1#2#3{\begin{scope}[shift={(#1)}, rotate=#2]
      \draw[#3,line width=1.4pt] (0,0) -- (-0.115, 0.085) -- (-0.115,-0.085) -- cycle;
      \draw[#3,line width=1.4pt] (0,0) -- ( 0.115, 0.085) -- ( 0.115,-0.085) -- cycle;
    \end{scope}}
  \def\bigon{%
    \draw[rd] (0,\H) .. controls (-\Wc,0) and (-\Wc,0) .. (0,-\H);
    \draw[bl] (0,\H) .. controls ( \Wc,0) and ( \Wc,0) .. (0,-\H);
    \draw[gy] (0,\H) -- (0,-\H);
    \notchat{(0,0.875)}{0}{biggrey}}

  \begin{scope}[shift={(0,0)}]
    \bigon
    \draw[lam] (0,0) -- (1.35, 0.865);
    \draw[lam] (0,0) -- (1.35,-0.865);
    \notchat{(0.53,0.337)}{123}{black}
    \node[dot] at (0,0) {};
    \node[lbl] at (0,\LY) {blue};
  \end{scope}

  \begin{scope}[shift={(\Px,0)}]
    \bigon
    \draw[lam] (0,0) -- (-1.35, 1.3);
    \draw[lam] (0,0) -- (-1.62, 0.385);
    \draw[lam] (0,0) -- (-1.67,-0.635);
    \notchat{(-0.72,-0.273)}{111}{black}
    \draw[lam] (-1.0,1.535) .. controls (1.45,0.6) and (1.45,-0.6) .. (-1.0,-1.535);
    \node[dot] at (0,0) {};
    \node[lbl] at (0,\LY) {red};
  \end{scope}

  \begin{scope}[shift={(2*\Px,0)}]
    \bigon
    \draw[lam] (-1.95,0.65) -- (1.95,0.65);
    \draw[lam] (-1.10,-1.935) .. controls (0,1.198) and (0,1.198) .. (1.10,-1.935);
    \node[dot] at (0,0) {};
    \node[lbl] at (0,\LY) {neutral};
  \end{scope}

\end{tikzpicture}
\caption{\label{fig:D bigon} Three classes of lamination curves that can pass through a bigon in type $D$.  The left contains two curves that contribute only to the {\color{bigblue}$\Delta$} portion of the mutation.  The middle contains many (not necessarily compatible) curves that contribute only to the {\color{bigred}$\Gamma$} portion of the mutation.  On the right, all lamination curves contribute equally to both {\color{bigblue}$\Delta$} and {\color{bigred}$\Gamma$} sums.}
\end{figure}
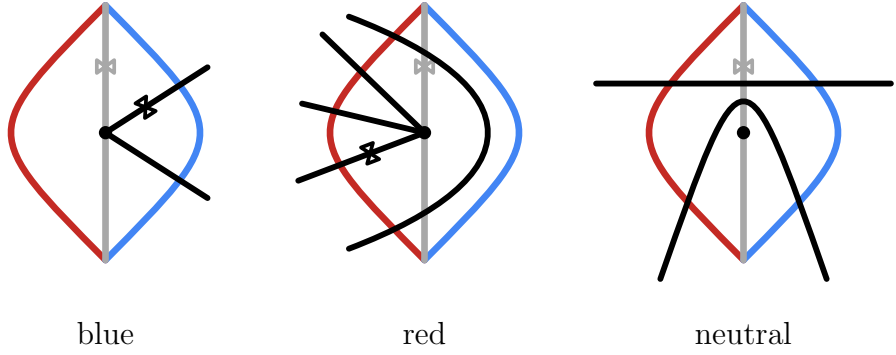

\begin{figure}
\centering
\definecolor{qdred}{RGB}{196,42,36}
\definecolor{qdblue}{RGB}{66,133,244}
\definecolor{qdgrey}{RGB}{168,168,168}
\begin{tikzpicture}[
    x=0.9cm, y=0.9cm,
    line cap=round, line join=round,
    rd/.style={line width=2.5pt, qdred},
    bl/.style={line width=2.5pt, qdblue},
    gy/.style={line width=2.2pt, qdgrey},
    lam/.style={line width=2.2pt, black},
    dot/.style={circle, fill=black, inner sep=0pt, minimum size=3.6pt},
    lbl/.style={font=\normalsize, anchor=base}
  ]

  \def\H{1.575}   
  \def\Bx{2.45}   
  \def\Pu{0.80}   
  \def\Pt{4.12}   
  \def\LY{-3.05}  
  \def\notchat#1#2#3{\begin{scope}[shift={(#1)}, rotate=#2]
      \draw[#3,line width=1.4pt] (0,0) -- (-0.105, 0.078) -- (-0.105,-0.078) -- cycle;
      \draw[#3,line width=1.4pt] (0,0) -- ( 0.105, 0.078) -- ( 0.105,-0.078) -- cycle;
    \end{scope}}
  \def\quadshape{%
    \draw[bl] (0, \H) .. controls (1.04, 1.80) and (2.06, 1.06) .. (\Bx,0);
    \draw[bl] (0,-\H) .. controls (1.04,-1.80) and (2.06,-1.06) .. (\Bx,0);
    \draw[gy] (0, \H) .. controls (0.45,-0.95) and (1.70,-1.05) .. (\Bx,0);
    \draw[gy] (0,-\H) .. controls (0.45, 0.95) and (1.70, 1.05) .. (\Bx,0);
    \draw[rd] (0,-\H) -- (0,\H);
    \draw[rd] (\Pu,0) .. controls (1.625, 0.48) and (1.625, 0.48) .. (\Bx,0)
        coordinate[pos=0.45] (nk);
    \draw[rd] (\Pu,0) .. controls (1.625,-0.48) and (1.625,-0.48) .. (\Bx,0);
    \notchat{nk}{97}{qdred}
    \node[dot] at (0,\H) {};   \node[dot] at (0,-\H) {};
    \node[dot] at (\Bx,0) {};  \node[dot] at (\Pu,0) {};}

  \begin{scope}[shift={(0,0)}]
    \quadshape
    \draw[lam] (0.52,2.30) -- (0.52,-2.25);
    \node[lbl] at (1.0,\LY) {blue};
  \end{scope}

  \begin{scope}[shift={(\Pt,0)}]
    \quadshape
    \draw[lam] (-0.30, 0.625) .. controls (1.817, 0.50) and (1.817,-0.50) .. (-0.30,-0.625);
    \draw[lam] (-0.30, 1.405) .. controls (1.30, 1.10) and (1.60,-0.40) .. ( 2.20,-1.305);
    \draw[lam] ( 2.20, 1.305) .. controls (1.60, 0.40) and (1.30,-1.10) .. (-0.30,-1.405);
    \node[lbl] at (1.0,\LY) {red};
  \end{scope}

  \begin{scope}[shift={(2*\Pt,0)}]
    \quadshape
    \draw[lam] ( 0.95, 2.055) .. controls (0.60, 1.45) and (0.30, 1.05) .. (-0.70, 0.825);
    \draw[lam] ( 2.6, 0.875) .. controls (1.95, 0.45) and (1.95,-0.45) .. ( 2.6,-0.875);
    \draw[lam] (-0.70,-0.825) .. controls (0.30,-1.05) and (0.60,-1.45) .. ( 0.95,-2.005);
    \node[lbl] at (1.0,\LY) {neutral};
  \end{scope}

\end{tikzpicture}
\caption{\label{fig:D colored laminations} Three classes of lamination curves that can pass through this quadrilateral in type $D$.  The left contains a single curve that contribute only to the {\color{qdblue}$\Delta$} portion of the mutation.  The middle contains several (not necessarily compatible) curves that contribute only to the {\color{qdred}$\Gamma$} portion of the mutation.  On the right, all lamination curves contribute equally to both {\color{qdblue}$\Delta$} and {\color{qdred}$\Gamma$} sums.}
\end{figure}
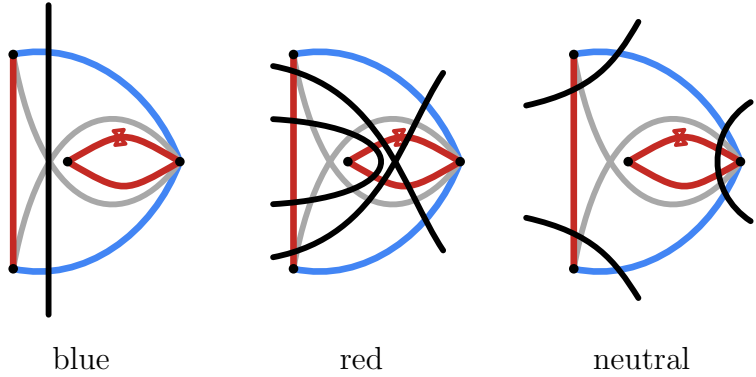

The following lemma states that in many instances, the Zamolodchikov fan associated to $D_n$ is closely related to the Zamolodchikov fan associated to $B_{n-1}$.
\begin{lemma}
\label{BD cones}
    Let $B$ be an (extended) exchange matrix for a bipartite orientation of type $D_{n}$ that folds to $B_{n-1}$.
    Let $f$ be the usual bicolored automorphism that folds $D_{n}$ to $B_{n-1}$, and let $f(B)$ be the resulting exchange matrix for $B_{n-1}$.
    Then, $\mathcal{Z}_B$ has the same number of cones as $\mathcal{Z}_{f(B)}$.
\end{lemma}
\begin{proof}
    If the vertices of $D_n$ are numbered
    \begin{align*}
        &2\\[-2mm]
        &\hspace{1mm}\rule{0.7pt}{1.6ex}\\[-2mm]
        1\hspace{1mm}-\hspace{1mm}&3-4-\cdots -n,
    \end{align*}
    then let $\lambda_1, \dots, \lambda_n$ be the corresponding tropical values of the cluster variables at a fixed time in the bipartite belt.
    
    Recall that mutation within a mutation quadrilateral as in Figure \ref{fig:D colored laminations} has the form
    \begin{align*}
        \lambda_3 + \lambda_3' = \max(\textcolor{red}{(\lambda_1 + \lambda_2) + \lambda_4}, \textcolor{blue}{y_1 + y_2}),
    \end{align*}
    where $y_1, y_2$ are the tropical labelings of some coefficients.
    Notice that this exchange relation is dependent only on $\lambda_1 + \lambda_2$, not the individual values of either label.
    Mutation within a bigon mutation quadrilateral as in Figure \ref{fig:D bigon} has the form $\lambda_k + \lambda_k' = \max(\textcolor{red}{\lambda_3}, \textcolor{blue}{y_1})$ for both $k = 1$ and $k = 2$, where $y_1$ is the tropical labeling of some coefficient.
    Thus the new sum is given by
    $$\lambda_1' + \lambda_2' = 2\max(\textcolor{red}{\lambda_3}, \textcolor{blue}{y_1}) - (\lambda_1 + \lambda_2),$$
    which is also completely determined by $\lambda_1 + \lambda_2$.
    So, the color sequence generated by $\lambda$ is dependent only on the sum $\lambda_1 + \lambda_2$.
    
    A labeling $\rho\in\mathbb{R}^{2n}$ is \textit{$f$-symmetric} if $\forall i\in [2n], \rho_{f(i)} = \rho_i$.
    Let $\rho$ be any initial labeling of $D_n$, and let $c\in\mathbb{R}$ be a constant scalar.
    Define a new labeling $\nu$ given by
    \begin{align*}
        \nu_i := \begin{cases}
            \frac{1}{2}(\rho_{1} + \rho_{2}) & i\in \{1, 2\}\\
            \rho_i & \text{otherwise}
        \end{cases},
    \end{align*}
    which averages the values of the radii.
    Then, $\nu$ generates the same sequence of red and blue mutations as $\rho$, and is now a $f$-symmetric labeling.
    By \cite[Lemma 4.11]{chin}, $\nu$ can be folded into a labeling $f(\nu)$ for $B_{n-1}$ with the same color sequence (up to identifying $\nu_{1}, \nu_2$). 

    Thus cones in $\mathcal{Z}_{f(B)}$ are equivalent to cones in $\mathcal{Z}_B$ generated by $f$-symmetric labelings.
    Since the $f$-symmetric labelings of $D_n$ cover all cones in $\mathcal{Z}_B$, the fans $\mathcal{Z}_B$ and $\mathcal{Z}_{f(B)}$ have the same number of cones.
\end{proof}

\begin{remark}
    Notice that under the conditions of this lemma, the initial $B$-matrix must be foldable into an exchange matrix for $B_{n-1}$, which ensures that it is not full rank.
    Thus one does not achieve the full number of cones possible for type $D_n$.
    The same argument can be used to show that the Zamolodchikov fans for $D_4$ and $G_2$ have the same number of cones, but does not work for all foldings.
    For example, folding $A_{2n+1}$ to $C_n$ and folding $E_6$ to $F_4$ result in significantly different numbers of cones.
\end{remark}

The following lemma is the analogous version of Lemma \ref{Type A: shifting labelings} for shifting type $D$ labelings by a constant.
\begin{lemma}
\label{Type D: shifting labelings}
    Let $B$ be the extended exchange matrix associated to the bipartite orientation of $D_n$ with tagged triangulation coefficients.
    Let $\lambda\in\mathbb{R}^{2n}$ be a tropical labeling of $B$, and let $c\in\mathbb{R}$ be a constant scalar.
    Define a new labeling $\rho\in\mathbb{R}^{2n}$ given by 
    \begin{align*}
        \rho_i := \begin{cases}
            \lambda_i + \frac{c}{2} & i = 1, 2,\\
            \lambda_i + c & \text{otherwise,}
        \end{cases}
    \end{align*}
    which shifts each radius by $\frac{c}{2}$ and each full diagonal by $c$.
    Then, $\lambda$ and $\rho$ generate the same colored mutation sequence in the bipartite belt.
\end{lemma}
\begin{proof}
    Recall the form for tropical mutation (\ref{eq:tropicalMutation}). 
    In type $D_n$, there are three cases to consider when mutating at $k$.
    When $k$ is a radius (i.e. $k\in\{1, 2\}$), then we get
    \begin{align*}
        \boldsymbol{t}_k^{\rho}(t+1) &= \max\left(\boldsymbol{t}_{i}^{\lambda}(t) + c, \boldsymbol{t}_{j}^{\lambda}(t) + c\right) - \left(\boldsymbol{t}_k^{\lambda}(t-1) + \frac{c}{2}\right)\\
        &= \max\left(\boldsymbol{t}_{i}^{\lambda}(t), \boldsymbol{t}_{j}^{\lambda}(t)\right) + c - \boldsymbol{t}_k^{\lambda}(t-1) - \frac{c}{2}\\
        &= \boldsymbol{t}_k^{\lambda}(t+1) + \frac{c}{2},
    \end{align*}
    where $i$ is the source of the incoming arrow, and $j$ is the target of the outgoing arrow.

    When $k$ is a non-radius diagonal, but is adjacent to a radius, then $\boldsymbol{t}_k^{\rho}(t+1)$ is equal to

    \begin{align*}
        &= \max\left(\big(\boldsymbol{t}_{i}^{\lambda}(t) + c\big) + \left(\boldsymbol{t}_{1}^{\lambda}(t) + \frac{c}{2}\right) + \left(\boldsymbol{t}_{2}^{\lambda}(t) + \frac{c}{2}\right), \boldsymbol{t}_{j_1}^{\lambda}(t) + \boldsymbol{t}_{j_2}^{\lambda}(t) + 2c\right) - \big(\boldsymbol{t}_k^{\lambda}(t-1) + c\big)\\
        &= \max\left(\boldsymbol{t}_{i}^{\lambda}(t) + \boldsymbol{t}_{1}^{\lambda}(t) + \boldsymbol{t}_{2}^{\lambda}(t), \boldsymbol{t}_{j_1}^{\lambda}(t) + \boldsymbol{t}_{j_2}^{\lambda}(t)\right) + 2c - \boldsymbol{t}_k^{\lambda}(t-1) - c\\
        &= \boldsymbol{t}_k^{\lambda}(t+1) + c,
    \end{align*}
    for some $i, j_1, j_2$ that are non-radius diagonals.

    Finally, when $k$ is a non-radius diagonal that borders two regular triangles (as in type $A$), then we get
    \begin{align*}
        \boldsymbol{t}_k^{\rho}(t+1) &= \max\Big(\big(\boldsymbol{t}_{i_1}^{\lambda}(t) + c\big) + \big(\boldsymbol{t}_{i_2}^{\lambda}(t) + c\big), \big(\boldsymbol{t}_{j_1}^{\lambda}(t) + c\big) + \big(\boldsymbol{t}_{j_2}^{\lambda}(t) + c\big)\Big) - \big(\boldsymbol{t}_k^{\lambda}(t-1) + c\big)\\
        &= \max\left(\boldsymbol{t}_{i_1}^{\lambda}(t) + \boldsymbol{t}_{i_2}^{\lambda}(t), \boldsymbol{t}_{j_1}^{\lambda}(t) + \boldsymbol{t}_{j_2}^{\lambda}(t)\right) + 2c - \boldsymbol{t}_k^{\lambda}(t-1) - c\\
        &= \boldsymbol{t}_k^{\lambda}(t+1) + c,
    \end{align*}
    where $i_1, i_2$ are the sources of incoming arrows, and $j_1, j_2$ are the targets of outgoing arrows.

    So for all $t\in\mathbb{Z}$ and for all $i\in [2n]$,
    \begin{align*}
        \boldsymbol{t}_i^{\rho}(t) &= \begin{cases}
            \boldsymbol{t}_i^{\lambda}(t) + \frac{c}{2} & i = 1, 2,\\
            \boldsymbol{t}_i^{\lambda}(t) + c & \text{otherwise.}
        \end{cases}
    \end{align*}
    $\boldsymbol{t}^{\rho}(t) = \boldsymbol{t}^{\lambda}(t) + c$ for all $t\in \mathbb{Z}$, and these two tropical $T$-systems have equivalent exchange relations in the bipartite belt.
    Thus a mutation in $\boldsymbol{t}^{\lambda}$ is blue (resp., red) if and only if it is blue (resp., red) in $\boldsymbol{t}^{\rho}$.
\end{proof}

Now we are ready to prove Theorem \ref{mainThm} for type $D$.

\begin{proposition}
\label{D}
    Let $B$ be the exchange matrix associated to the bipartite orientation of $D_n$ with tagged triangulation coefficients.
    In one period, there are $n(2n - 2)$ red mutations and $2n$ blue mutations.
\end{proposition}
\begin{proof}

\begin{figure}
\centering
\definecolor{lamblue}{RGB}{38,90,160}
\begin{tikzpicture}[
    x=1.15cm, y=1.15cm,
    line cap=round, line join=round,
    bdry/.style={line width=1.7pt, black},
    arc/.style={line width=1.7pt, black},
    lam/.style={line width=1.7pt, lamblue},
    dot/.style={circle, fill=black, inner sep=0pt, minimum size=3.6pt},
    wt/.style={lamblue, font=\small, inner sep=0.5pt}
  ]

  \def\Pt{5.9}    
  \def\notchat#1#2#3{\begin{scope}[shift={(#1)}, rotate=#2]
      \draw[#3,line width=1.1pt] (0,0) -- (-0.085, 0.065) -- (-0.085,-0.065) -- cycle;
      \draw[#3,line width=1.1pt] (0,0) -- ( 0.085, 0.065) -- ( 0.085,-0.065) -- cycle;
    \end{scope}}
  \def\dnine{%
    \coordinate (L)  at (-1,1);
    \foreach \k in {0,1,2,3}{
      \coordinate (T\k) at (\k,2);
      \coordinate (B\k) at (\k,0);}
    \coordinate (P) at (3.42,1);
    \draw[bdry] (L) -- (T0) -- (T1) -- (T2) -- (T3)
                arc[start angle=90, end angle=-90, radius=1]
                -- (B2) -- (B1) -- (B0) -- cycle;
    \draw[arc] (T0) -- (B0);  \draw[arc] (T1) -- (B0);  \draw[arc] (T1) -- (B1);
    \draw[arc] (T2) -- (B1);  \draw[arc] (T2) -- (B2);  \draw[arc] (T3) -- (B2);
    \draw[arc] (T3) -- (B3);
    \draw[arc] (B3) .. controls (3.259,0.48) and (3.259,0.48) .. (P);
    \draw[arc] (B3) .. controls (3.579,0.345) and (3.579,0.345) .. (P)
        coordinate[pos=0.2] (nk);
    \notchat{nk}{124}{black}
    \foreach \v in {L,T0,T1,T2,T3,B0,B1,B2,B3,P}{\node[dot] at (\v) {};}}

  \begin{scope}[shift={(0,0)}]
    \dnine
    \draw[lam] (1.05,2.35) .. controls (1.85,1.58) and (3.05,0.55) .. (4.38,0.36);
    \foreach \x/\y in {1.48/2.20, 1.52/1.40, 1.88/1.1, 2.3/0.9,
                       2.86/0.58, 3.12/0.52, 3.36/0.43, 4.1/0.66}{
      \node[wt] at (\x,\y) {\scalebox{0.8}{$1$}};}
  \end{scope}

  \begin{scope}[shift={(\Pt,0)}]
    \dnine
    \draw[lam] (1.05,2.35) .. controls (1.85,1.58) and (2.80,1.05) .. (P)
        coordinate[pos=0.85] (bk);
    \notchat{bk}{75}{lamblue}
    \foreach \x/\y in {1.48/2.20, 1.52/1.40, 1.88/1.1, 2.3/0.9, 2.86/0.58}{
      \node[wt] at (\x,\y) {\scalebox{0.8}{$1$}};}
    \node[wt] at (3.2,0.7) {\scalebox{0.8}{$0$}};   
    \node[wt] at (3.62,0.6) {\scalebox{0.8}{$1$}};   
  \end{scope}

\end{tikzpicture}
\caption{\label{fig:D laminations} Two different type $D$ lamination curves of weight $1$.  On the left is a standard curve contributing to 8 diagonals: two frozen variables and six cluster variables.  On the right is a tagged lamination radius contributing to all non-radius diagonals it crosses, as well as the single tagged radius.}
\end{figure}
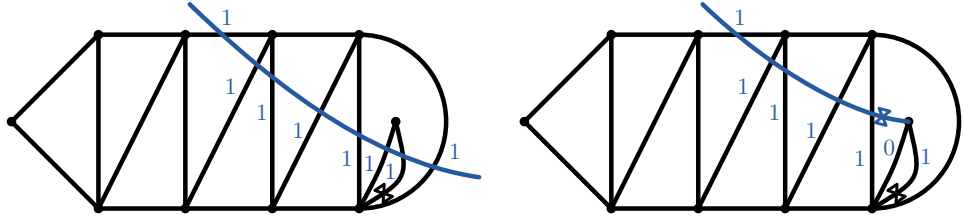

    Consider the tropical vector $v_{\ell}$ corresponding to a single lamination curve $\ell$ with weight 1.
    A non-radius lamination curve will contribute 1 to every diagonal that it intersects, and 0 otherwise.
    A radius lamination curve will contribute 1 to every non-radius diagonal that it intersects, 1 to any radius tagged the same way that it shares a puncture with, and 0 otherwise (see Figure \ref{fig:D laminations}).
    For a specific mutation quadrilateral, a lamination curve can determine the color to be blue or red (or neither).
    Notice for any fixed quadrilateral, there is only one lamination curve that can make the mutation blue, whereas there are potentially many lamination curves that can make the mutation red (see Figures \ref{fig:D bigon} and \ref{fig:D colored laminations}).
    By Lemma \ref{Type D: shifting labelings}, it suffices to only consider positive tropical labelings.

    Given such a maximal lamination, we can uniquely identify the color of each mutation.
    If our mutation quadrilateral is a bigon, by definition (\ref{sec:laminations}) there cannot exist both blue and red lamination radii.
    So, the radii will determine the color of this mutation.
    If there are no radii, then a blue radius can be added that is compatible with the lamination, contradicting the maximality of the original lamination.
    If the quadrilateral is not a bigon, any blue lamination curve is incompatible with a red lamination curve, so these curves will determine the color of the mutation.
    If there are no colored curves, a blue lamination curve can be added that is compatible with the lamination, contradicting the maximality of the original lamination.

    From this construction, it is clear that the number of blue mutations in one half period is the size of the maximal lamination (a pair of lamination radii will form two blue mutations, one for each radius in the tagged triangulation).
    By half-periodicity, each cluster variable appears exactly once in each half of the bipartite belt. 
    So the number of blue mutations in one half period is $n$, the number of nontrivial curves in the maximal lamination.
    In one period, there are $2n$ blue mutations, and the remaining $n(2n-2)$ mutations are red.

    As stated earlier, these laminations will not necessarily cover all positive tropical labelings.
    However, we show that these laminations still cover every chamber in the Zamolodchikov fan.
    Since the blue mutations are exactly determined by each curve in the maximal lamination, these maximal laminations will each correspond to distinct cones in $\mathcal{Z}$.
    Counting maximal laminations is equivalent to choosing a source for both radii to connect to, and then counting the number of triangulations of an $(n+1)$-gon:
    \begin{align*}
        n\cdot C_{n-1} &= n\cdot \frac{1}{n}{2n-2\choose n-1} = {2n-2\choose n-1}.
    \end{align*}
    
    By Proposition \ref{BD cones}, the total number of cones will be the same as the total number of cones of $f(B)$ after folding to $B_{n-1}$.
    As $f(B)$ has full rank, by Corollary \ref{chamber count}, the number of cones is ${2n-2\choose n-1}$, the number of seeds in the cluster algebra of type $B_{n-1}$.
    So type $D_{n}$ with tagged triangulations only contains ${2n-2\choose n-1}$ cones, all of which correspond to a maximal lamination with the correct number of red and blue mutations.
\end{proof}
\subsection{Homomesy in type $B$ tropical $T$-systems}
Once again, we can use a standard folding argument to obtain type $B_n$ from type $D_{n+1}$.
\begin{corollary}
    Let $B$ be the full rank exchange matrix associated to the bipartite orientation of $B_n$ with tagged triangulation coefficients.
    In one period, there are $2n^2$ red mutations and $2n$ blue mutations.
\end{corollary}
\begin{proof}
    As mentioned in the above proof of Proposition \ref{D}, the type $D_{n+1}$ laminations are already foldable into type $B_n$, and the number of maximal laminations matches the correct number of type $B_n$ seeds.
    We need only check that each of these maximal laminations produces the correct color count for type $B_n$.
    
    Identify the two cluster variable radii together, as they become folded together in type $B_n$.
    We can now view mutating the radii together as one combined mutation, as they will always be the same color.
    Now count the number of blue mutations.
    A maximal lamination consists of $2$ radii and $n-1$ full curves.
    Every non-radius lamination curve is blue for exactly one quadrilateral, which contributes one blue mutation.
    And the two lamination radii are blue for the same bigon which encompasses both cluster variable radii, contributing one blue mutation.
    Thus, the total number of blue mutations in one half period is $n$.

\end{proof}

\section{Types $E_6, E_7, E_8$, $F_4$, and $G_2$} \label{sec:E}

The exceptional types $E_6, E_7, E_8, F_4$, and $G_2$ have been checked via a computer search as follows:
\begin{itemize}
    \item[(1)] Add coefficients to make the $B$-matrix full rank.
    \item[(2)] By Proposition \ref{chamber count}, the number of cones in $\mathcal{Z}$ is $c_{\Lambda}$, the number of seeds in the cluster algebra of type $\Lambda$.
    \item[(3)] Randomly sample $\lambda$, an initial tropical labeling from the rational numbers, and execute one period of mutations in the bipartite belt, coloring each mutation red or blue.
    \item[(4)] If the labeling is generic (every mutation is colored), verify the number of blue mutations is $2n$, then store the color sequence in a dictionary. 
    \item[(5)] Repeat from step (3) until $c_{\Lambda}$ different color sequences are discovered.
\end{itemize}

\begin{figure}
\begin{center}
\begin{tabular}{ |c|c|c|c|c|c|c|c|c|c| }
    \hline
    & $A_n$ & $B_n$ & $C_n$ & $D_n$ & $E_6$ & $E_7$ & $E_8$ & $F_4$ & $G_2$\\
    \hline
    Seeds & $\frac{1}{n+2}{2n+2\choose n+1}$ & ${2n\choose n}$ & ${2n\choose n}$ & $\frac{3n-2}{n}{2n-2\choose n-1}$ & 833 & 4160 & 25080 & 105 & 8\\
    \hline
    Verified Cones & $\frac{1}{n+2}{2n+2\choose n+1}$ & ${2n\choose n}$ & ${2n\choose n}$ & ${2n-2\choose n-1}$ & 833 & 4160 & 25080 & 105 & 8\\
    \hline
    Roots & $n(n+1)$ & $2n^2$ & $2n^2$ & $n(2n-2)$ & 72 & 126 & 240 & 48 & 12\\
    \hline
\end{tabular}
\end{center}
\caption{\label{table} The number of seeds, verified cones, and roots in the associated root system for each finite type cluster algebra. Type $D_n$ was verified only for a not full rank $B$-matrix, so it does not achieve every seed.}
\end{figure}

See Figure \ref{table} for the number of cones for each exceptional type.

This completes the proof of Theorem \ref{mainThm}.

\begin{remark}
    The original conjecture \ref{conj} was stated in the hopes of a bijective proof between roots and red mutations. 
    For types $A_n, B_n, C_n$, and $D_n$, the red mutations can be understood from the root system perspective as follows. 
    Each mutation quadrilateral appears exactly once in each half-period, and contains a unique nontrivial lamination curve. 
    A maximal lamination $L$ is a (tagged) triangulation, which can be viewed as a cluster in the associated cluster algebra.
    Then all of the nontrivial lamination curves in $L$ are cluster variables in the initial seed, which also correspond to the blue mutation quadrilaterals. 
    Every nontrivial lamination curve not in $L$ corresponds to a red mutation quadrilateral, and a cluster variable not in the initial seed. 
    Due to \cite{fomin2}, there is a natural bijection between cluster variables outside of the initial seed and the positive roots in the root system, given by the denominator vectors. 
    Thus once a maximal lamination has been fixed, there is a natural bijection between positive roots and red mutation quadrilaterals in the first half-period. 
    The second half-period can then be interpreted to represent the corresponding negative roots.
    While the conjecture has also been proved computationally for the exceptional types ($E_6, E_7, E_8, F_4, G_2$), we were unable to find an analogous bijection.
\end{remark}

\bibliographystyle{alpha}
\bibliography{sources}

\end{document}